\documentclass[a4paper,12pt,reqno]{amsart}
\usepackage[lmargin=2.4cm,rmargin=2.25cm,bmargin=3cm,tmargin=2.5cm]{geometry}

\usepackage[utf8]{inputenc}
\usepackage{amsmath} 
\usepackage{amssymb}
\usepackage{amsthm}

\usepackage{graphicx}
\usepackage{xcolor}
\usepackage{hyperref}
\usepackage{enumerate}

\hypersetup{colorlinks}

\newtheorem{theorem}{Theorem}[section]

\newtheorem{proposition}[theorem]{Proposition}
\newtheorem{lemma}[theorem]{Lemma}

\theoremstyle{definition}
\newtheorem{definition}[theorem]{Definition}

\theoremstyle{remark}
\newtheorem{remark}[theorem]{Remark}

\newcommand{\stable}[2]{\mathcal{S}_{#1}( #2 )}
\newcommand{\udomain}{\mathcal{U}}
\newcommand{\law}{\mathcal{L}}
\newcommand{\dfn}{\mathrel{\mathop:}=} 
\newcommand{\nfd}{=\mathrel{\mathop:}}
\newcommand{\real}{\mathbb{R}}
\newcommand{\integer}{\mathbb{Z}}

\newcommand{\nanu}{\mathbb{N}}
\newcommand{\pmspace}{\mathcal{P}}
\newcommand{\E}{\mathbb{E}}
\newcommand{\ud}{\mathrm{d}}
\renewcommand{\P}{\mathbb{P}}
\renewcommand{\vec}[1]{\boldsymbol{#1}}
\newcommand{\given}{\,:\,}
\newcommand{\bigmid}{\;\big|\;}

\newcommand{\arXiv}[1]{\href{https://arxiv.org/abs/#1}{arXiv{:}#1}}
\newcommand{\doi}[1]{\href{https://doi.org/#1}{DOI{:}#1}}

\begin{document}

\title[Wide stable neural networks]{Wide stable neural networks: Sample regularity, functional convergence and {B}ayesian inverse problems}

\author{Tomás Soto}
\address{School of Engineering Sciences, LUT University}
\email{Tomas.Soto@lut.fi}

\begin{abstract} 
We study the large-width asymptotics of random fully connected neural networks with weights drawn from $\alpha$-stable distributions, a family of heavy-tailed distributions arising as the limiting distributions in the Gnedenko-Kolmogorov heavy-tailed central limit theorem. We show that in an arbitrary bounded Euclidean domain $\mathcal{U}$ with smooth boundary, the random field at the infinite-width limit, characterized in previous literature in terms of finite-dimensional distributions, has sample functions in the fractional Sobolev-Slobodeckij-type quasi-{B}anach function space $W^{s,p}(\mathcal{U})$ for integrability indices $p < \alpha$ and suitable smoothness indices $s$ depending on the activation function of the neural network, and establish the functional convergence of the processes in the space of probability measures on $W^{s,p}(\mathcal{U})$. This convergence result is leveraged in the study of functional posteriors for edge-preserving Bayesian inverse problems with stable neural network priors.
\end{abstract} 


\maketitle
\sloppy

\section{Introduction}

The study of scaling limits of wide random neural networks was initiated by R.~Neal in \cite{RMN96}, in the context of priors for Bayesian learning. Neal used the classical central limit theorem to show, roughly speaking, that a perceptron with one hidden layer of width $H \in \nanu$ and suitable assumptions (such as that of finite variance) on the weight distributions and the activation function, converges with respect to finite-dimensional marginal distributions to a Gaussian process as $H \to \infty$.

The Gaussian process behaviour at the infinite-width limit has since been studied for a variety of neural network architectures, and different modes of convergence, in an expansive collection of papers by numerous authors. Motivations, apart from the context of Bayesian learning proposed by Neal, include the understanding of different initialization schemes for training, the formal study of gradient descent based training dynamics via the so-called neural tangent kernel at the infinite-width limit, and quantitative estimation of the convergence rates. As examples we mention the papers \cite{MHRTG18,JGH18-NTK,LXSBNSP19,yang20,hanin2023,EMT21,BT23,BFFS21-functional,FHMNP23}, although this list is necessarily incomplete.

The focus of this paper is in the non-Gaussian regime, where we consider heavy-tailed distributions for the weight distributions, resulting with suitable assumptions and scaling in an \emph{$\alpha$-stable process} at the infinite-width limit. These neural networks have been considered as priors for edge-preserving Bayesian inversion in \cite{LDS22}, and recently it has been observed in \cite{GSZ21} that heavy-tail like behavior for the network weights can also arise naturally in deep learning via gradient descent training.

In Sections \ref{ss:stable}--\ref{ss:nn} below, we recall the relevant concepts and definitions, and present a brief overview of literature pertaining to the heavy-tailed regime. The main results are presented in Section \ref{ss:results}. Section \ref{ss:structure} contains an overview of the paper's structure.

\subsection{Stable distributions}\label{ss:stable}

Most random variables considered in this paper will have a \emph{L\'evy $\alpha$-stable} distribution. We recall here their basic properties, and mention \cite{ST94} as the standard reference for both univariate and multivariate stable distributions.

We are in particular interested in symmetric $\alpha$-stable distributions, which are characterized in terms of a \emph{stability index} $\alpha \in (0,2]$ and a \emph{scale parameter} $\sigma > 0$, in the sense that if the random variable $u$ has an $\alpha$-stable distribution, its characteristic function is given by
\begin{equation}\label{eq:stable-characteristic}
  \E\left[ \exp\left( i \theta u\right)\right] = \exp\left( -\left(\sigma |\theta|\right)^\alpha\right)
\end{equation}
for all $\theta \in \real$. In this case, we write
\[
  u \sim \stable{\alpha}{\sigma}.
\]

From \eqref{eq:stable-characteristic} it is routinely seen that $2$-stable distributions are zero-mean Gaussians, and we will largely ignore this edge case in the sequel. Besides the case $\alpha = 2$, a closed-form expression for the density function for a random variable $u$ distributed as \eqref{eq:stable-characteristic} is known only for $\alpha = 1$, in which case $u$ has a standard Cauchy distribution. However, for all $\alpha \in (0,2)$, $u$ is a heavy-tailed distribution, and denoting by $p$ and $F$ its density and cumulative density functions respectively, we have
\[
  p(x) \sim c_{\alpha,\sigma} x^{-\alpha-1}
  \quad
  \text{and}
  \quad
  1-F(x) \sim c'_{\alpha,\sigma} x^{-\alpha}
\]
as $x \to \infty$ for some constants $c_{\alpha,\sigma}$, $c'_{\alpha,\sigma} > 0$. This obviously means that for $p > 0$ we have $\E[|u|^p] < \infty$ if and only if $p < \alpha$.

By the generalized heavy-tailed central limit theorem (see e.g.~\cite[\S 35, Theorem 2]{gnedenko-kolmogorov}), the symmetric $\alpha$-stable distributions are precisely the heavy-tailed distributions arising as the limiting distributions of suitably scaled sums of i.i.d symmetric random variables with heavy power-law-like tails.

An immediate consequence of \eqref{eq:stable-characteristic} (and in fact an alternative definition for symmetric $\alpha$-stable distributions) is that if $u_1$, $u_2$, $\cdots$, $u_n$ and $u$ are independently distributed as $\stable{\alpha}{\sigma}$, then
\[
  \sum_{i=1}^n \theta_i u_i \stackrel{d}{=} \left( \sum_{i=1}^n |\theta_i|^\alpha\right)^{1/\alpha} u \nfd \|\vec{\theta}\|_{l^\alpha} u
\]
for all $\vec{\theta} \dfn (\theta_1,\cdots,\theta_n) \in \real^n$, generalizing the familiar property from Gaussian distributions.

Multivariate symmetric $\alpha$-stable distributions can be defined in a similar manner, with a \emph{spectral measure} $\Lambda$ instead of the scale parameter $\sigma$. Since we will not need to directly access the spectral measure of any multivariate stable distribution in this paper, let us simply recall \cite[Theorem 2.1.5]{ST94} that a random vector in $\real^d$ is symmetric $\alpha$-stable if and only if all linear combinations of its components have a univariate symmetric $\alpha$-stable distribution.

Finally, a stochastic process $(f(x))_{x \in \mathcal{X}}$ with any (possibly uncountable) indexing set $\mathcal{X}$ is said to be symmetric $\alpha$-stable if $(f(x_1), \cdots, f(x_n)) \in \real^n$ is symmetric $\alpha$-stable for all finite $(x_1,\cdots,x_n) \subset \mathcal{X}$.

\subsection{The model and previous literature}\label{ss:nn}

Write $\alpha \in (0,2)$ for the stability index, fixed from now. For the sake of notational convenience, we will for the most part consider shallow neural networks with a single hidden layer (extensions of the main results to deeper architectures are presented in Section \ref{se:deeper}) and real output. We thus define the network $f^H$ of width $H \in \nanu$ as
\[
  f^H(x) \dfn \frac{1}{H^{1/\alpha}} \sum_{i=1}^H v_i \varphi(u_i^{\intercal} x + a_i) + b 
\]
for $x \in \real^d$, where the $v_i \in \real$, $b \in \real$, the coordinates of $u_i \in \real^d$ and $a \in \real$ are all indepedently distributed as $\stable{\alpha}{\sigma_v}$, $\stable{\alpha}{\sigma_b}$, $\stable{\alpha}{\sigma_u}$ and $\stable{\alpha}{\sigma_a}$ respectively, and $\varphi \colon \real\to\real$ is a continuous activation function.

The scaling factor $H^{-1/\alpha}$ above is chosen in the spirit of the heavy-tailed central limit theorem \cite[\S 35, Theorem 2]{gnedenko-kolmogorov}, as in it is in essence what one would expect to give rise to a scaling limit as $H \to \infty$ in case $\varphi$ were a bounded function.

The basic question of large-width asymptotics then concerns the existence of a limiting process as $H\to \infty$, e.g.~whether there exists an $\alpha$-stable process $(f^\infty(x))_{x\in\real^d}$ such that 
\begin{equation}\label{eq:convergence-finite-dims}
  \bigl( f^H(x_1), f^H(x_2), \cdots, f^H(x_n)\bigr) \stackrel{d}{\longrightarrow} \bigl( f^\infty(x_1), f^\infty(x_2), \cdots, f^\infty(x_n)\bigr)
\end{equation}
as $H \to \infty$ for all $\{x_1, x_2, \cdots, x_n\} \subset \real^d$.

This convergence was already heuristically explored by Neal in \cite{RMN96}, and formally it was established for shallow neural networks in \cite{DL06} with certain assumptions (such as essentially sub-linear growth) on the activation function $\varphi$. Further results for deep neural networks, convolutional architectures, more general heavy-tailed network weights, and e.g.~ReLU-activations have recently been obtained in \cite{FFP23-deep-stable,JLLY23,BFF22-deep-stable,FFP22-relu,BFFP22,LAJLYC23}.

Due to the heavy-tailed nature of the weights, the sample functions of $f^H$ are known \cite{LDS22} to have discontinuity-like properties with high probability at moderate to large widths, even with a bounded and smooth activation function such as $\tanh$ and even with a single hidden layer as in our model, making them an attractive choice of a functional prior for Bayesian inverse problems where the goal is to model rough features or edges \cite{SCR22}.

In Figures \ref{fi:1d} and \ref{fi:2d} we have plotted simulated sample functions for input dimensions $d \in \{1, 2\}$ with different stability indices $\alpha$, chosen such that the characteristic jump-like behaviour (or lack of thereof) can be observed in $[-1,1]$ and $[-1,1]^2$ respectively. The simulations were done with $(\sigma_v, \sigma_b, \sigma_u, \sigma_a) = (1, 0, 5, 2)$, $H = 10^5$ and
$
  \varphi = x \mapsto x/\max(|x|,1).
$
As a side remark and without going into detail, the expected locations of the ``jumps'' of the sample functions of shallow heavy-tailed neural networks is an interesting question as noted in \cite[Section 2.2]{LDS22}, and it can be controlled essentially by modifying the distribution of the input biases $a_i$. The main results of this paper continue to hold with a variety of such modifications.

\begin{figure}
  \includegraphics[width= \textwidth]{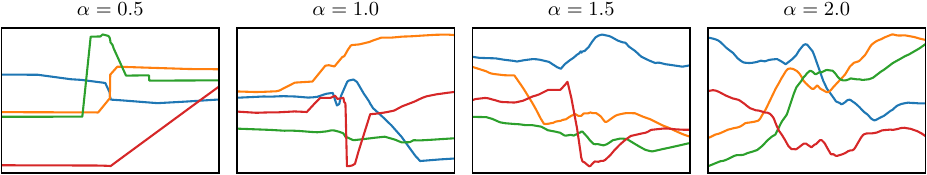}
  \caption{Sample paths of wide neural networks with one hidden layer on $[-1,1]$}\label{fi:1d}
\end{figure}

\begin{figure}
  \includegraphics[width= \textwidth]{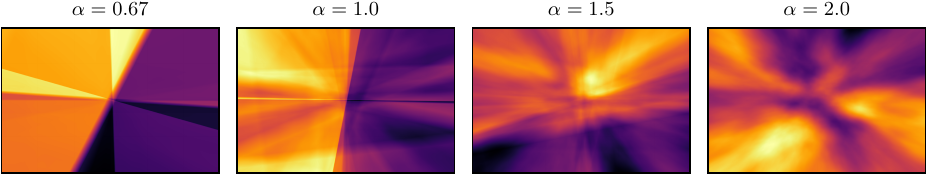}
  \caption{Sample functions of wide neural networks with one hidden layer on $[-1,1]^2$}\label{fi:2d}
\end{figure}

\subsection{Main results and discussion}\label{ss:results}

As our first main result, we establish an infinite-di\-men\-sio\-nal variant of the convergence \eqref{eq:convergence-finite-dims}. This involves determining a suitable function space containing the sample functions of both $f^H$ and $f^{\infty}$ with full probability. This function space will then act as an ambient space for the Bayesian posterior analysis in our second main result. 

The analysis is purely theoretical in nature, and we refer to the recent paper \cite{LDS22} for numerical examples of Bayesian inversion with this kind of priors. Our two main theorems together can be seen as a sort of \emph{discretization invariance} result (see e.g.~\cite[Section 3]{RPL13} and the references therein) for wide heavy-tailed neural network priors, in the sense that widening a properly scaled finite-width neural network will bring both the prior and the resulting posterior closer to well-defined limit processes. Our results are also tangentially related to the literature concerned with the construction of continuous-parameter random fields for Bayesian inversion. We refer to e.g.~\cite{LSS09, KLSS23, AW24, SCR22, CLR19, KLS22} for related examples of Besov, Laplace, $\alpha$-stable and non-standard Gaussian priors.

Our first main result is Theorem \ref{th:besov-convergence} below. The second main result of this paper, concerning Bayesian posterior convergence, requires some technical preparations to be stated in its full generality, so this will be postponed until Section \ref{se:bayes-ip}, but we present a preliminary version of it in Remark \ref{re:posterior-continuous} right after Theorem \ref{th:besov-convergence}. This will be followed by Remarks \ref{re:main-theorem-remarks}--\ref{re:parameter-ranges} offering some commentary on the assumptions, parameter ranges and the choice of the function space in Theorem \ref{th:besov-convergence}. An extension of the main results to deep neural networks is presented in Section \ref{se:deeper}.

In order to state the main results, we first introduce a relevant class of function spaces. Let $\udomain \subset \real^d$ be a domain. The fractional Sobolev-Slobodeckij space $W^{s,p}(\udomain)$ for $s \in (0,1)$ and $p > d/(d+s)$ is defined as the collection of Lebesgue-measurable functions
 $f \colon \udomain \to \real$ such that
 \begin{equation}\label{eq:slobodeckij-norm}
  \|f\|_{W^{s,p}(\udomain)} \dfn \|f\|_{L^p(\udomain)} + \left( \int_{\udomain \times \udomain} \frac{|f(x)-f(y)|^p}{|x-y|^{sp+d}} \ud x \ud y\right)^{1/p}
\end{equation}
is finite.

\begin{theorem}\label{th:besov-convergence}
Let $\udomain$ be a bounded domain in $\real^d$ with $C^{\infty}$-smooth boundary (open interval in case $d = 1$). Assume that the activation function $\varphi$ is H\"older-continuous and uniformly sublinearly increasing, i.e.~that there exist $\lambda \in (0,1]$ and $\beta \in [0,1)$ such that
  \[
    |\varphi(x) - \varphi(y)|
    \leq
    \begin{cases}
      c_\varphi |x - y|^{\lambda} \qquad & \text{if } |x - y| < 1,
      \\
      c'_\varphi |x-y|^{\beta} & \text{if } |x - y| \geq 1
    \end{cases}
  \]
  for some constants $c_\varphi$, $c'_\varphi < \infty$. Assume that $\alpha \in (d/(d+\lambda),2)$.

  Then
  \begin{enumerate}[(i)]
    \item\label{en:thm-besov-integrability} 
    For any integrability exponent $p \in (d/(d+\lambda),\alpha)$ and smoothness index $s \in (0,\lambda)$ such that $p > d/(d+s)$,
    there exists a version of the limit process $(f^{\infty}(x))_{x \in \udomain}$ in \eqref{eq:convergence-finite-dims} with sample functions in the fractional Sobolev-Slobodeckij space $W^{s,p}(\udomain)$;
  
    \item\label{en:thm-law-convergence} 
    Denoting by $\law(f^{H})$, $H \in \nanu\cup \{\infty\}$, the law of the process $f^{H}$ in part (\ref{en:thm-besov-integrability}), 
    we have the convergence
    \[
      \law(f^{H}) \stackrel{w}{\longrightarrow} \law(f^{\infty})
    \]
    in the space of Borel probability measures on $W^{s,p}(\mathcal{U})$ as $H \to \infty$.
  \end{enumerate}
\end{theorem}

We refer to Remark \ref{re:besov-integrability-remarks} (\ref{en:random-besov-element}) in Section \ref{se:sample-regularity} for the precise mathematical meaning of ``process with sample functions in a function space'', as in part (\ref{en:thm-besov-integrability}) above.

\begin{remark}\label{re:posterior-continuous}
  From the above Theorem we immediately have the following functional result on Bayesian posterior convergence, in the spirit of \cite[Proposition 1]{HBNPS-2020}. Let $\mathcal{G} \colon W^{s,p}(\udomain) \to \real^M$ be a continuous forward operator. This could be a collection of local averages, or more generally integrations against test functions in $L^q(\udomain)$ where
    \[
      q = \frac{dp}{(d+s)p - d} > 1
    \]
    in case $sp < d$, or a vector-valued function constructed from the afore-mentioned operations and continuous functions on $\real$. Then we consider the Bayesian inverse problem
    \[
      \vec{u} = \mathcal{G}(f) + \vec{\varepsilon}_M,
    \]
    where $\vec{\varepsilon}_M$ is an independent $\real^M$-valued noise term (for example multivariate Gaus\-sian). For given observations $\vec{u}$, the likelihood $\pi(\vec{u} \bigmid \cdot\,)$ is under suitable mild assumptions a bounded and continuous function on $W^{s,p}(\udomain)$, and so the posterior induced by the prior $\law(f^H)$, $H \in \nanu\cup\{\infty\}$, can be characterized as
    \[
      \pi^H(F \bigmid \vec{u}) = \frac{\gamma^H(F \bigmid \vec{u})}{\gamma^H(1 \bigmid \vec{u})}
    \]
    for bounded and continuous test functions $F \colon W^{s,p}(\udomain) \to \real$, where the unnormalized posterior $\gamma^H(\, \cdot \bigmid \vec{u})$ is given by
    \[
      \gamma^H(F \bigmid \vec{u}) = \int_{W^{s,p}(\udomain)} F(f) \pi(\vec{u} \bigmid f) \law(f^H)(\ud f).
    \]
    The integrand above is again under suitable assumptions a bounded and continuous function on $W^{s,p}(\udomain)$, so we have the convergence
    \[
      \pi^H(\, \cdot \bigmid \vec{u}) \stackrel{w}{\longrightarrow} \pi^\infty(\, \cdot \bigmid \vec{u})
    \]
    as $H \to \infty$. We refer to Section \ref{se:bayes-ip} for details, and a posterior convergence result (Theorem \ref{th:bayes-ip}) for more general forward operators.
\end{remark}

\begin{remark}\label{re:main-theorem-remarks}
  \begin{enumerate}[(i)]
    \item The Sobolev-Slobodeckij spaces $W^{s,p}(\udomain)$ are Besov-type function spaces of fractional-order smoothness that arise e.g.~in the contexts of trace theorems for standard integer-order Sobolev spaces and real interpolation between Sobolev and Lebesgue spaces. We refer to Appendix \ref{ap:properties-besov} for a short survey of properties of functions with this sort of regularity.

    \item The space $W^{s,p}(\udomain)$ is a quasi-Banach space when $p < 1$, in the sense that the triangle inequality holds only up to a multiplicative constant. The applicability of quasi-Banach function spaces for edge preserving Bayesian inversion has previously been observed in \cite{sullivan2017}.

    \item The spaces $W^{s,p}(\udomain)$ were chosen as the framework for our analysis because they are reasonably well-known and their definition is easy to state and understand. However, because the parameter range for $(p,s)$ in Theorem \ref{th:besov-convergence} is open-ended, we can replace $W^{s,p}(\udomain)$ in the statement of the Theorem with a variety of function spaces with similar smoothness and integrability properties.

    For instance, we have \cite[Theorem 3.3.1]{triebel-I}
    \[
      W^{s+\epsilon, p}(\udomain) \subset B^s_{p,q}(\udomain)
      \quad \text{and} \quad
      W^{s+\epsilon, p}(\udomain) \subset F^s_{p,q}(\udomain)
    \]
    for all $q > d/(d+s)$ and (small enough) $\epsilon > 0$ with continuous embeddings, so the Theorem continues to hold for the Besov space $B^s_{p,q}(\udomain)$ and the Triebel-Lizorkin space $F^s_{p,q}(\udomain)$ with admissible incides. As an example, denoting by $\mathcal{F}$ the standard Fourier transform on $\real^d$, the fractional Hardy-Sobolev space of functions $f \in L^1_{\mathrm{loc}}(\real^d)$ satisfying
    \[
      \mathcal{F}^{-1}\bigl[\xi \mapsto (1 + |\xi|^2)^{s/2} (\mathcal{F}f)(\xi)\bigr]
      \in 
      \begin{cases}
        L^p(\real^d), \quad & p > 1, \\
        h^p(\real^d), \quad & p \leq 1
      \end{cases}
    \]
    restricted to $\udomain$ coincides with the space $F^s_{p,2}(\udomain)$ \cite[Theorems 2.3.8 and 2.5.8/1]{triebel-I}. For $p \leq 1$, $h^p(\real^d)$ above stands for the real-variable local Hardy space.

    \item The restriction $\beta < 1$ precludes the use of the popular ReLU activation function in the Theorem. This is not a deficiency of the result, but a reflection of the fact that the question of $\alpha$-stable network convergence is genuinely different for linearly increasing activation functions, and it may hold (if at all) under a different scaling regime. In the context of ReLU activation in particular, this has recently been investigated in \cite{FFP22-relu}.
  
    \item \label{en:intro-kolmogorov-continuity}
    In case $d = 1$ and $\alpha \lambda > 1$, the sample paths of $f^{\infty}$ are in fact H\"older-continuous and a stronger mode of convergence holds; see Remark \ref{re:measurability-remarks} (\ref{en:kolmogorov-continuity}).
  \end{enumerate}
\end{remark}

\begin{remark}\label{re:parameter-ranges}
  A couple further comments on the parameter range in Theorem \ref{th:besov-convergence}.

  \begin{enumerate}[(i)]
    \item
    The upper bound $\lambda$ for the smoothness index $s$ is essentially sharp. Starting from the definition of the space $W^{\lambda, p}(\udomain)$, it is not hard to construct an example of an admissible $\lambda$-H\"older activation function $\varphi$ that does not belong to $W^{\lambda,p}(\real)$ locally, and the sample functions of any $f^H$ do not belong to $W^{s,p}(\udomain)$ with positive probability.

    \item
    The upper bound $\alpha$ for $p$ comes essentially from the application of Fubini's theorem, and the fact that $\alpha$-stable random variables are $p$-integrable precisely for $p < \alpha$. We do not know whether this is the best possible integrability range for the $s$-smoothness for the sample functions of $f^{\infty}$ for given $s$. However, in general we do not expect the permissible range of $ps$ to extend beyond $\alpha\lambda$. For if this were the case for $(d,\alpha,\lambda) = (1,1,1)$, then the infinite-width limit of Cauchy neural networks on the real line would have H\"older-continuous sample paths (see Remark \ref{re:main-theorem-remarks} (\ref{en:intro-kolmogorov-continuity}) above), which appears empirically unlikely.

    \item
    The lower bound $d/(d+\lambda)$ for $p$ (and consequently for $\alpha$) becomes increasingly restrictive for large $d$, but it always admits the case of Cauchy distributions with $\alpha = 1$.

    It is a consequence of the fact that, in order to work with quasi-Banach function spaces with integrability and smoothness $(p,s)$ measured in terms of local integrals, we basically need the lower bound $p > d/(d + s)$ to quarantee local integrability, and this condition also directly shows up in our proofs in a natural manner.

    For $p < d/(d+s)$, it is possible to define Besov-type spaces such as $W^{s,p}(\udomain)$ in a variety of ways, but the wide unity of different characterizations in the literature breaks down. For example, the Fourier-analytical definition of these spaces \cite{triebel-I} will in this case contain tempered distributions that do not admit representations as functions. A more direct definition, like the one preceding the statement of the Theorem, will result in a space of functions that does not embed in the space of tempered distributions, complicating the analysis of smoothness properties considerably \cite[Section 2.2.3]{triebel-I}.

    Recently \cite{HKT-2017} approximation results based on local medians instead of local integral averages have been obtained for Besov-type spaces for the full range $p > 0$, but this approach does not appear to be easily applicable in our setting, where Fubini's theorem for local integrals plays a key role in many proofs. More advanced results, such as fractional Rellich-Kondrachov-type compactness theorems, are also not known in this context as far as we are aware.
  \end{enumerate}
\end{remark}

\subsection{Structure of the paper and notation}\label{ss:structure}

The paper is structured as follows. In Section \ref{se:sample-regularity} we dig into the sample regularity of the processes $f^H$, $H \in \nanu$, verifying certain measurability properties (implicit in the statement of Theorem \ref{th:besov-convergence}) and approximation results in the process, culminating in the proof of Theorem \ref{th:besov-convergence}.

In Section \ref{se:bayes-ip} we review some technicalities concerning functional Bayesian inverse problems, already alluded to in Remark \ref{re:posterior-continuous} above, and state and prove the second main theorem of this paper, Theorem \ref{th:bayes-ip}.

In Section \ref{se:deeper} we extend Theorems \ref{th:besov-convergence} and \ref{th:bayes-ip} for arbtirarily deep neural networks in the infinite-width limit, under the additional assumption that the activation function $\varphi$ is Lipschitz-continuous (i.e.~$\lambda = 1$).

In Appendix \ref{ap:properties-besov} we review standard properties of the function spaces $W^{s,p}(\udomain)$, which are extensively used in many of our proofs. Appendix \ref{ap:auxiliary} contains the proofs of some auxiliary technical results used in Sections \ref{se:sample-regularity}--\ref{se:bayes-ip}.

We end this section by introducing some notation conventions. Denote by $\nanu \dfn \{1,2,3,\cdots\}$ the set of positive integers, and by $a{:}b$ the discrete interval $\{a, a+1, \cdots, b-1, b\}$ for any $a$, $b \in \integer$ with $a < b$. The standard Lebesgue measure on $\real^d$ is denoted by $m_d$, and we use the shorthand $\ud x$ for Lebesgue integration with respect to $m_d$.

Probabilities and expectations (with respect to probability measures obvious from context) are denoted by $\P$ and $\E$ respectively. A random process indexed by a set $\mathcal{X}$, living in a probability space $\Omega$, will be denoted by both $(f(x))_{x \in \mathcal{X}}$ and $(f(\omega,x))_{(\omega,x) \in (\Omega,\mathcal{X})}$ depending on the context, with no room for confusion.

The space of Borel probability measures on a metric space $Z$ is denoted by $\pmspace(Z)$. The notations
\[
  \stackrel{w}{\longrightarrow}
  \quad
  \text{and}
  \quad
  \stackrel{d}{\longrightarrow}
\]
are used to signify weak convergence of measures in $\pmspace(Z)$ and convergence in distribution of random elements in $Z$ respectively.

We will use the symbols $\lesssim$, $\gtrsim$ and $\approx$ when dealing with unimportant multiplicative constants. That is, when $f$ and $g$ are non-negative functions with the same domain, the notations $f \lesssim g$ or $g \gtrsim f$ mean that there exists a positive constant $c$, independent of some parameters, so that $f \leq c g$. We will emphasize the parameters with respect to which $f \lesssim g$ holds if they are not obvious from context. The notation $f \approx g$ means that $f \lesssim g$ and $g \lesssim f$.

\section{Sample regularity and functional convergence}\label{se:sample-regularity}

The goal here is to establish regularity results (and conditions for them) for the sample functions of the limit process in \eqref{eq:convergence-finite-dims}. Although we have talked about ``the'' limit process, we only have access to the process in terms of finite-dimensional distributions, and it is a general fact of probability theory that the fine properties of the sample functions of a continuous-parameter process are not uniquely determined by its finite-dimensional distributions.

In order to talk about regularity properties measured in terms of local integrals, we will reconcile with a version of the limit process that admits measurable sample functions.

\begin{proposition}\label{pr:measurable-sample-functions}
Suppose that the activation function $\varphi$ satisfies the assumptions of Theorem \ref{th:besov-convergence} with $\lambda \in (0,1]$ and $\beta \in [0,1)$, and that $\alpha \in (0,2)$.
\begin{enumerate}[(i)]
  \item \label{en:modulus-of-continuity}
  For any $p \in (0, \alpha)$, we have the following uniform modulus of continuity in probability:
  \[
    \sup_{H \in \nanu \cup \{ \infty\} }\E\left[ |f^H(x) - f^H(y)|^p \right]
    \leq
    \begin{cases}
      C_p |x-y|^{\lambda p} \quad & \text{if } \lambda < 1
      \\
      C_p |x-y|^{p} \left(|\log|x-y|| + 1\right)^{p/\alpha} & \text{if } \lambda = 1
    \end{cases}
  \]
  for all $x$, $y \in \real^d$ with $|x - y| \leq 1$, with a constant $C_p$ independent of $x$ and $y$.

  \item \label{en:measurable-version}
  For any $H \in \nanu \cup \{\infty\}$, there exists a version of the process $f^H$ defined on a complete probability space $(\Omega^H, \mathcal{S}^H, \P^H)$ such that the mapping
  \[
    \Omega^H \times \real^d \owns (\omega, x) \mapsto f^H(\omega, x) \in \real \cup \{ -\infty, \infty \}
  \]
  is jointly measurable with respect to the product $\sigma$-field.
\end{enumerate}
\end{proposition}

\begin{remark}\label{re:measurability-remarks}
\begin{enumerate}[(i)]
  \item The condition $|x-y| \leq 1$ in part (\ref{en:modulus-of-continuity}) above can be replaced with $|x-y| \leq R$ for any $R > 0$, affecting only the constant $C_p$. This obvious from the proof below.

  \item Part (\ref{en:measurable-version}) is, of course, obvious for $H < \infty$, in which case $f^H(x)$ is a continuous function of the input $x$ and a finite number of continuously distributed parameters.

  \item We a priori allow the version of $f^{\infty}$ in part (\ref{en:measurable-version}) to take values in the two-point compactification $\real \cup \{ -\infty, \infty \}$ of $\real$ in order to appeal to general results concerning the existence of measurable versions. This will be refined later in Proposition \ref{pr:besov-integrability}.

  \item \label{en:kolmogorov-continuity}
  In case $d = 1$ and $\lambda \alpha > 1$, part (\ref{en:modulus-of-continuity}) together with a general version of the Kolmogorov-Chentsov continuity theorem (see e.g.~\cite[Theorem 23.7]{kallenberg}) implies that the limit process $f^{\infty}$ is almost surely H\"older continuous with exponent $r$ for any $r \in (0,\lambda - 1/\alpha)$, and the $d$-convergence $f^H \longrightarrow f^{\infty}$ takes place in the space $C_{\real}(\real)$ of continuous functions with the topology generated by local $\sup$-norms. When $d \geq 2$ we are far from this regime however.
\end{enumerate}
\end{remark}

\begin{proof}
(\ref{en:modulus-of-continuity}) 
We begin by noting that the case $H = \infty$ will follow from the uniform bound for the finite-width networks. For $j \in \nanu$, write $\Lambda^j$ for the function $x \mapsto \min(|x|, j)$. Then $(z, z') \mapsto \Lambda^j(z - z')^p$ is a bounded and continuous function on $(\real^d)^2$, so by monotonicity and the convergence \eqref{eq:convergence-finite-dims},
\begin{align*}
  \E\left[ |f^\infty(x) - f^\infty(y)|^p \right]
  & = \lim_{j \to \infty} \lim_{H \to\infty} \E\left[ \Lambda^j( f^H(x) - f^H(y) )^p \right]
  \\
  & \leq \liminf_{H \to \infty} \E\left[ |f^H(x) - f^H(y)|^p \right].
\end{align*}

Let then $H < \infty$. In the rest of this proof we write $\varphi_i \dfn \varphi(u_i^{\intercal} \cdot + a_i)$ for $i \in 1{:}H$. Using the freezing lemma in conjunction with the tower property of conditional expectations and the $\alpha$-stability property of the weights $v_i$, we have
\begin{align*}
  \E\left[ |f^H(x) - f^H(y)|^p \right]
  & = \E\left[  \E\left[ \left| \sum_{i=1}^{H} \frac{\varphi_i(x)-\varphi_i(y)}{H^{1/\alpha}} v_i \right|^{p} \bigmid \left\{ \varphi_i(x), \, \varphi_i(y) \right\}_{i \in 1{:}H} \right]  \right]
  \\
  & = \E\left[ |v_1|^p \right] \E\left[ \left(\sum_{i=1}^{H} \frac{|\varphi_i(x)-\varphi_i(y)|^\alpha}{H} \right)^{p/\alpha} \right].
\end{align*}
Noting that $\E[|v_1|^p] \approx \sigma_v^p \approx 1$ and using H\"older's inequality (with exponent $\alpha/p > 1$) in the latter expectation, we get
\begin{equation}\label{eq:modulus-holder}
  \E\left[ |f^H(x) - f^H(y)|^p \right] \lesssim \E\left[ |\varphi_1(x) - \varphi_1(y)|^\alpha \right]^{p/\alpha},
\end{equation}
where the right-hand side and the implicit multiplicative constant are independent of $H$.

Now since $\varphi_1(x) - \varphi_1(y) = \varphi(u_1^{\intercal}x + a_1) - \varphi(u_1^{\intercal}y + a_1)$, the assumptions on the activation $\varphi$ imply
\[
  |\varphi_1(x) - \varphi_1(y)|
  \leq
  \begin{cases}
    c_\varphi |u_1^{\intercal}(x - y)|^{\lambda} \qquad & \text{if } |u_1^{\intercal}(x - y)| < 1,
    \\
    c'_\varphi |u_1^{\intercal}(x - y)|^{\beta} & \text{if } |u_1^{\intercal}(x - y)| \geq 1,
  \end{cases}
\]
and the right-hand side estimates above are independent of the input bias $a_1$.
For fixed $x$, $y \in \real^d$, the random variable $\|x-y\|_{l^\alpha}^{-1} u_1^{\intercal}(x - y)$ is by assumption symmetric and $\alpha$-stable with scale parameter $\sigma_u$. Write $q$ for its density function. We have $q(u) \lesssim 1\land |u|^{-\alpha-1}$ for $u \in \real\setminus\{0\}$, with a multiplicative constant depending only on $\alpha$ and $\sigma_u$.

Thus,
\begin{align*}
  & \E\left[ |\varphi_1(x) - \varphi_1(y)|^{\alpha} \boldsymbol{1}\left(\|x-y\|_{l^\alpha}^{-1} |u_1^{\intercal}(x - y)| \geq \|x-y\|_{l^\alpha}^{-1} \right) \right]
  \\
  & \qquad \lesssim \|x-y\|_{l^\alpha}^{\alpha\beta} \int_{|u| \geq \|x-y\|_{l^\alpha}^{-1}} u^{\alpha\beta} u^{-\alpha-1} \ud u
  \lesssim \|x-y\|_{\ell^\alpha}^{\alpha}
  \approx |x-y|^{\alpha}
\end{align*}
with multiplicative constants independent of $x$, $y \in \real^d$. On the other hand,
\begin{align*}
  & \E\left[ |\varphi_1(x) - \varphi_1(y)|^\alpha \boldsymbol{1}\left(\|x-y\|_{l^\alpha}^{-1} |u_1^{\intercal}(x - y)| < \|x-y\|_{l^\alpha}^{-1}\right) \right]
  \\
  & \qquad \lesssim \|x-y\|_{\ell^\alpha}^{\alpha\lambda} \int_{|u| < \|x-y\|_{l^\alpha}^{-1}} u^{\alpha\lambda} q(u) \ud u
  \lesssim
  \begin{cases}
    |x-y|^{\alpha \lambda} \quad & \text{if } \lambda < 1,
    \\
    |x-y|^{\alpha} \left( |\log|x-y|| + 1 \right) & \text{if } \lambda = 1,
  \end{cases}
\end{align*}
again with constants independent of $x$ and $y$.

Overall,
\[
  \E\left[ |\varphi_1(x) - \varphi_1(y)|^\alpha  \right] \lesssim |x-y|^{\lambda \alpha}\left( |\log|x-y|| + 1 \right), 
\]
where the logarithmic term can be discarded in case $\lambda < 1$. Applying this in \eqref{eq:modulus-holder} yields the desired estimate.

(\ref{en:measurable-version}) Part (\ref{en:modulus-of-continuity}) readily implies that the process $f^H$ is for every $H \in \nanu \cup \{\infty\}$ continuous in probability in the usual sense:
\[
  \P\left( |f^H(x) - f^H(y)| > \varepsilon \right) \leq \frac{C_p}{\varepsilon^p}\E[ |f^H(x) - f^H(y)|^p ] \longrightarrow 0
\]
as $|x-y| \to 0$ for all $\varepsilon > 0$. This is a well-known criterion for the existence of a measurable version of $f^H$, possibly taking values in the two-point compactification of $\real$; see e.g.~\cite[Chapter IV, Section 3, Theorem 1 and Remarks 1--2]{gikhman-skorokhod}.
\end{proof}

In the sequel, we write $\udomain$ for a fixed bounded domain in $\real^d$ with $C^\infty$-smooth boundary (open interval in case $d = 1$). This boundary regularity condition is imposed mostly for the sake of technical convenience -- we believe that a Lipschitz domain such as the unit cube would also work, but were unable to find suitable references for the analysis of quasi-Banach function spaces in this context.

Recall the fractional Sobolev-Slobodeckij space $W^{s,p}(\udomain)$ for $s \in (0,1)$ and $p > d/(d+s)$, defined as the collection of Lebesgue-measurable functions $f \colon \udomain \to \real$ such that the quasinorm  \eqref{eq:slobodeckij-norm} in Section \ref{ss:results} is finite (more precisely, $W^{s,p}(\udomain)$ becomes a quasi-normed space after identifying the functions that agree $m_d$-almost everywhere on $\udomain$). 

Although $\|\cdot\|_{W^{s,p}(\udomain)}$ is not a norm for $p < 1$, $W^{s,p}(\udomain)$ is a separable and complete metric space when endowed with the metric $(f, g) \mapsto \|f - g\|_{W^{s,p}(\udomain)}^{p \land 1}$. A Sobolev-type embedding theorem implies that in the range $p > d/(d+s)$, functions in $W^{s,p}(\udomain)$ are in fact locally integrable, and that evaluations of local averages are continuous functionals in $W^{s,p}(\udomain)$. For these and other results concerning these function spaces, see Proposition \ref{pr:besov-properties} in Appendix \ref{ap:properties-besov}.

We then have following result concerning the expected smoothness energy of the sample functions of $f^H$, including the limiting process.

\begin{proposition}\label{pr:besov-integrability}
Suppose that the activation function $\varphi$ satisfies the assumptions of Theorem \ref{th:besov-convergence} with $\lambda \in (0,1]$ and $\beta \in [0,1)$, and that $\alpha \in (d/(d+\lambda), 2)$.

\begin{enumerate}[(i)]
  \item\label{en:besov-integrability}
  For any $p \in (d/(d+\lambda), \alpha)$ and $s \in (0,\lambda)$ such that $p >
  d/(d+s)$,
  \[
    \sup_{H \in \nanu \cup \{\infty\}} \E\left[ \|f^H\|_{W^{s,p}(\udomain)}^p \right] < \infty.
  \]

  \item\label{en:besov-process}
  For every $H \in \nanu \cup \{\infty\}$, and $p$ and $s$ like in part (\ref{en:besov-integrability}), an indistinguishable version of the process $f^{H}$ can be realized as a process with sample functions in $W^{s,p}(\udomain)$, in the sense that $f^{H}(\omega, \cdot) \in W^{s,p}(\udomain)$ for all $\omega \in \Omega^H$, and $\omega \mapsto f^{H}(\omega, \cdot)$ is an $\mathcal{S}^H$-measurable random element in the complete and separable metric space $W^{s,p}(\udomain)$.
\end{enumerate}

\end{proposition}

\begin{remark}\label{re:besov-integrability-remarks}
  \begin{enumerate}[(i)]
    \item
    In part (\ref{en:besov-integrability}) above, the process $f^H$ is to be understood as a measurable version as in part (\ref{en:measurable-version}) of Proposition \ref{pr:measurable-sample-functions}, restricted to $x \in \udomain$. This version could in principle take values in $\{ -\infty, \infty \}$, presenting a problem for the definition of the $W^{s,p}(\udomain)$-quasinorm defined in terms of first-order differences. However, the first part of the proof below shows that $|f^H(x)| < \infty$ for almost all $x \in \udomain$ almost surely. With this and the measurability of the process $f^H$ in mind, the statement of part (\ref{en:besov-integrability}) is well-defined, and Fubini's theorem implies that the exchanges of the order of integration in the proof are justified.

    \item\label{en:random-besov-element}
    Regarding part (\ref{en:besov-process}), we stress that from this point on, it is understood that $f^H$ exists both as an $\mathcal{S}^H$-measurable pointwise defined process $(f^H(x))_{x \in \udomain}$, and as an $\mathcal{S}^H$-measurable random function as in the statement of the result. These are a priori two distinct (but not mutually exclusive) ways for a process to exist, in the sense that pointwise evaluation of functions is not a strictly well-defined operation in the function space $W^{s,p}(\udomain)$.
  \end{enumerate}
\end{remark}

\begin{proof}
(\ref{en:besov-integrability}) 
We first consider the expected $L^p$-energy of $f^H$ for $H < \infty$ (the uniform estimate for $H < \infty$ can be extended to $H = \infty$ like in the proof of Proposition \ref{pr:measurable-sample-functions}). By Fubini's theorem and the freezing lemma, we have
\begin{align*}
  \E\left[ \|f^H\|_{L^p(\udomain)}^p \right]
  & = \E\left[ \int_{\udomain} |f^H(x)|^p \ud x  \right]
  = \int_{\udomain} \E\left[|f^H(x)|^p \right] \ud x
  \\
  & = \int_{\udomain} \E[|v_1|^p] \E\left[ \left(\sum_{i=1}^{H} \frac{|\varphi(u^\intercal_i x + a_i)|^\alpha}{H} + \left(\frac{\sigma_b}{\sigma_v}\right)^\alpha \right)^{p/\alpha}\right] \ud x
  \\
  & \lesssim \int_{\udomain} \E\left[ |\varphi(u_1^{\intercal} x + a_1)|^\alpha \right]^{p/\alpha} \ud x + 1
\end{align*}
as in the proof of Proposition \ref{pr:measurable-sample-functions}. The upper bound above is independent of $H$, and the integral over $\udomain$ is finite because of the sublinearity condition on $\varphi$. Namely, $u_1^{\intercal} x + a_1$ has by assumption a symmetric $\alpha$-stable distribution with scale parameter $(\sigma_u^{\alpha} \| x \|_{\ell^\alpha}^\alpha + \sigma_a^\alpha)^{1/\alpha}$, and $|\varphi(y)| \leq c''_\varphi(1 + |y|^\beta)$ for some constant $c''_\varphi$. Thus,
\[
  \E\left[ |\varphi(u_1^{\intercal} x + a_1)|^\alpha \right] \lesssim \E\left[ 1 + |u_1^{\intercal} x + a_1|^{\alpha\beta} \right] \lesssim 1 + (\sigma_u^{\alpha} \| x \|_{\ell^\alpha}^\alpha + \sigma_a^\alpha)^\beta
\]
for some implicit multiplicative constants independent of $x$. In fact, the latter quantity can be bounded independently of $x$ since the domain $\udomain$ is bounded by assumption.

Since $f^H \in L^p(\udomain)$ almost surely, we obviously have $|f^H(x)| < \infty$ for almost all $x \in \udomain$ almost surely. Thus we can consider homogeneous smoothness part of \eqref{eq:slobodeckij-norm}, which we denote by $\|f^H(x)\|_{\dot{W}^{s,p}(\udomain)}$. We have, again by Fubini's theorem,
\begin{align*}
  \E\left[ \|f^H(x)\|_{\dot{W}^{s,p}(\udomain)}^p \right]
  & = \E\left[ \int_{\udomain \times \udomain}  \frac{|f^H(x)-f^H(y)|^p}{|x-y|^{s p+d}} \ud x \ud y \right]
  \\
  & = \int_{\udomain \times \udomain} |x-y|^{-s p-d} \E\left[ |f^H(x)-f^H(y)|^p \right] \ud x \ud y.
\end{align*}
Applying the modulus of continuity estimate in Proposition \ref{pr:measurable-sample-functions} (\ref{en:modulus-of-continuity}) here yields
\[
  \E\left[ \|f^H(x)\|_{\dot{W}^{s,p}(\udomain)}^p \right]
  \lesssim \int_{\udomain \times \udomain} |x-y|^{(\lambda-s) p-d} \left|\log|x-y|\right|^{p/\alpha} \ud x \ud y,
\]
where the upper bound is independent of $H$ (and finite because $\lambda - s > 0$).

(\ref{en:besov-process}) 
By part (\ref{en:besov-integrability}) and the completeness of the probability space $(\Omega^H, \mathcal{S}^H, \P^H)$ (see part (\ref{en:measurable-version}) in Proposition \ref{pr:measurable-sample-functions}), we can redefine $f^H(\omega, \cdot) \equiv 0$ for the events $\omega$ such that $\|f(\omega, \cdot)\|_{W^{s,p}(\udomain)}$ is $\infty$ or undefined (in the sense of $|f^H(\omega, x) - f^H(\omega, y)|$ being undefined in case $f^H(\omega, z) \in \{ -\infty, \infty\}$ for $z$ in a positive-measure subset of $\udomain$), the totality of which by part (\ref{en:besov-integrability}) has zero $\P^H$-measure.

It then remains to show that under the mapping $\omega \mapsto f^H(\omega, \cdot)$, the preimages of closed balls in $W^{s,p}(\udomain)$ belong to $\mathcal{S}^H$. Under the properties of the process $f^H$ accumulated thus far, this result is not particular to the neural networks under consideration, and is most easily proven using additional results on the topology of $W^{s,p}(\udomain)$, so we will postpone the proof until Appendix \ref{ap:auxiliary}.
\end{proof}

With the above result in mind, we can introduce the next major ingredient for the proof of Theorem \ref{th:besov-convergence}.

\begin{proposition}\label{pr:relatively-compact}
Let $p$ and $s$ be like in Theorem \ref{th:besov-convergence}. Write $\law(f^H)$ for the law of the process $f^H$ on $W^{s,p}(\udomain)$ (see part (\ref{en:besov-process}) of Proposition \ref{pr:besov-integrability}). Then the set
\[
  \left\{ \law(f^H) \given H \in \nanu \right\}
\]
is relatively compact in the space $\pmspace(W^{s,p}(\udomain))$.
\end{proposition}

\begin{proof}
We can obviously find $p' < p$ and $s' \in (s, \lambda)$ so that $p' > d/(d+s')$ and $s' - d/p' < s - d/p$. Then a Rellich-Kondrachov-type embedding theorem (see e.g.~Proposition \ref{pr:besov-properties} (\ref{en:rellich-kondrachov}) in Appendix \ref{ap:properties-besov}, where the roles of $(s,p)$ and $(s',p')$ are reversed) implies that we have
\[
  W^{s',p'}(\udomain) \subset W^{s,p}(\udomain)
\]
with a compact embedding. In particular
\[
  K_C \dfn \{ f \in W^{s,p}(\udomain) \given \|f \|_{W^{s',p'}(\udomain)} \leq C \}
\]
is a compact set in $W^{s,p}(\udomain)$ for all $C > 0$, and by Proposition \ref{pr:besov-integrability} (\ref{en:besov-integrability}), we have
\[
  \sup_{H \in \nanu} \P\left( f^H \in W^{s,p}(\udomain) \setminus K_C \right) \leq C^{-p'} \sup_{H \in \nanu } \E\left[ \|f^H\|_{W^{s',p'}(\udomain)}^{p'} \right] \longrightarrow 0.
\]
as $C \to \infty$. Prohorov's theorem \cite[Theorem 5.1]{billingsley} thus implies the stated compactness.
\end{proof}

For the sake of stating the next results, we introduce the following notation for local averages. Write
\[
  f_A \dfn \frac{1}{m_{\udomain}(A)} \int_{A} f(x) \ud x
\]
for all $f \colon \udomain \to \real$ and $A \subset \udomain$ such that the right-hand side above is well-defined (in particular, it is for $f \in L^1(\udomain)$ and measurable $A \subset \udomain$ with $m_d(A) > 0$).

\begin{proposition}\label{pr:convergence-local-averages}
Under the hypotheses of Theorem \ref{th:besov-convergence}, the family of processes $(f^H)_{H \in \nanu}$ converges to $f^{\infty}$ in terms of local averages, in the sense that
\begin{equation}\label{eq:convergence-local-avgs}
  \left( (f^H)_{B_1}, \cdots, (f^H)_{B_N} \right)
  \stackrel{d}{\longrightarrow} \left( (f^{\infty})_{B_1}, \cdots, (f^{\infty})_{B_N} \right)
\end{equation}
as $H \to \infty$ for all finite collections $(B_i)_{1 \leq i \leq N}$ of balls centered in $\udomain$.
\end{proposition}

For the proof, we need the following uniform convergence result for discrete approximation of local integrals. Its proof is given in Appendix \ref{ap:auxiliary}.

\begin{lemma}\label{le:integral-approximations}
Let $\alpha$, $\beta$, $\lambda$ and $p$ be as in Proposition \ref{pr:besov-integrability}. Then for any measurable subset $A$ of $\udomain$ with nonzero measure, and any $n \in \nanu$, there exist points $(x^n_j)_{j=1}^{\kappa(n)} \subset \udomain$ and scalars $(\theta^n_j)_{j=1}^{\kappa(n)} \subset \real_+$ independent of $H \in \nanu \cup \{\infty\}$ such that
\[
  \lim_{n\to\infty} \E \left[ \Biggl| (f^H)_A - \sum_{j=1}^{\kappa(n)} \theta^n_j f^{H}(x^n_j) \Biggr|^{p} \right] = 0
\]
uniformly with respect to $H$.

\end{lemma}

\begin{proof}[Proof of Proposition \ref{pr:convergence-local-averages}]
Write $\mathcal{B}$ for the collection of balls in \eqref{eq:convergence-local-avgs}, and $(f^H)_{\mathcal{B}}$ for the vector on the left-hand side of \eqref{eq:convergence-local-avgs}. It suffices to show that for arbitrary scalars $\nu \dfn (\nu_i)_{i=1}^N \subset \real^N$ and an arbitrary bounded Lipschitz-continuous function $g \colon \real \to \real$, it holds that
\begin{equation} \label{eq:convergence-local-avgs-lipschitz}
  \lim_{H \to \infty} \E\left[ g\left( \nu^{\intercal}(f^H)_{\mathcal{B}}\right) \right] = \E\left[ g\left( \nu^{\intercal}(f^\infty)_{\mathcal{B}}\right) \right].
\end{equation}

For all $i \in \{1, 2, \cdots, N\}$, write $X^{n}_i \dfn (x^{n}_{j,i})_{j=1}^{\kappa(n,i)}$ for the collection of points given by Proposition \ref{pr:local-averages} for $A \dfn B_i$, and $\theta^{n}_i$ for the collection (vector) of scalars given by the same Proposition. Here and in the rest of this proof, we have abused notation by writing $f^H(X^{n}_i)$ for the vector obtained by applying $f^H$ pointwise in $X^{n}_i$.

We can then estimate
\begin{align}
  & \left| \E\left[ g\left( \nu^{\intercal}(f^H)_{\mathcal{B}}\right) \right] - \E\left[ g\left( \nu^{\intercal}(f^\infty)_{\mathcal{B}}\right) \right]\right|
  \notag
  \\
  \leq \; & \E\left[ \left|  g\left( \nu^{\intercal}(f^H)_{\mathcal{B}}\right) - g\left( \sum_{i=1}^{N}\nu_i (\theta^{n}_i)^\intercal f^H(X^{n}_i)\right)\right| \right] 
  \label{eq:convergence-local-averages-H}
  \\
  & \qquad + \E\left[ \left| g\left( \nu^{\intercal}(f^\infty)_{\mathcal{B}}\right) - g\left( \sum_{i=1}^N \nu_i (\theta^{n}_i)^\intercal f^\infty(X^{n}_i)\right)\right|\right] 
  \label{eq:convergence-local-averages-inf}
  \\
  & \qquad + \left| \E\left[ g\left( \sum_{i=1}^{N}\nu_i (\theta^{n}_i)^\intercal f^H(X^{n}_i)\right)\right] - \E\left[ g\left( \sum_{i=1}^{N}\nu_i (\theta^{n}_i)^\intercal f^\infty(X^{n}_i)\right)\right] \right|.
  \label{eq:convergence-local-averages-finite-dim}
\end{align}

By the assumptions on $g$, we have $|g(x) - g(y)| \leq c_{g} |x-y|^{p \land 1}$ for all $x$, $y \in \real$ and some constant $c_{g} > 0$. Thus for any given $\epsilon > 0$, we can use subadditivity and Proposition \ref{pr:local-averages} to find an approximation level $n$ such that the terms \eqref{eq:convergence-local-averages-H} and \eqref{eq:convergence-local-averages-inf} are $< \epsilon/2$ for all $H \in \nanu \cup \{\infty\}$. For this $n$, the term \eqref{eq:convergence-local-averages-finite-dim} converges to zero as $H \to \infty$ by the convergence \eqref{eq:convergence-finite-dims}. This establishes \eqref{eq:convergence-local-avgs-lipschitz}.
\end{proof}

We are now ready to present the proof of the main result of this paper.

\begin{proof}[Proof of Theorem \ref{th:besov-convergence}]
Part (\ref{en:thm-besov-integrability}) of the result is contained in Proposition \ref{pr:besov-integrability} (see also Remark \ref{re:besov-integrability-remarks} (\ref{en:random-besov-element})).

Concerning part (\ref{en:thm-law-convergence}), we recall from Proposition \ref{pr:convergence-local-averages} that the local averages of $f^H$ convergence in distribution to those of $f^{\infty}$ as $H \to \infty$. By Proposition \ref{pr:local-averages} in Appendix \ref{ap:auxiliary}, the collection of local averages is a separating family for $\pmspace(W^{s,p}(\udomain))$, so it follows that any $w$-convergent subsequence of $(\law(f^H))_{H \in \nanu}$ must converge to $\law(f^{\infty})$. This together with Proposition \ref{pr:relatively-compact} then implies that the sequence $(\law(f^H))_{H \in \nanu}$ must itself converge to $\law(f^{\infty})$.
\end{proof}

\section{Bayesian posterior analysis}\label{se:bayes-ip}

We return to the question of Bayesian posterior convergence, which we briefly discussed in Remark \ref{re:posterior-continuous} in the context of forward operators that are continuous on $W^{s,p}(\udomain)$. In this section we will refine this result so as to accommodate forward operators based on pointwise evaluations, which are not continuous and not even strictly well-defined operators in $W^{s,p}(\udomain)$.

Recall the basic setting, i.e.~the Bayesian inverse problem
\begin{equation}\label{eq:bayes-ip}
  \vec{u} = \mathcal{G}(f) + \vec{\varepsilon}_M,
\end{equation}
where $\mathcal{G}\colon W^{s,p}(\udomain) \to \real^M$ is the forward operator and $\vec{\varepsilon}_M$ is an independent $\real^M$-valued noise term (such as component-wise i.i.d Gaussian). Let $\pi \in \pmspace(W^{s,p}(\udomain))$ be a prior distribution for $f$. Under fairly general conditions (see e.g.~\cite[Theorem 6.31]{Stuart_2010} or \cite[Theorem 3.3]{Lasanen-2012-II}), the posterior distribution $\pi(\, \cdot \bigmid \vec{u}) \in \pmspace(W^{s,p}(\udomain))$ for given observations $\vec{u} \in \real^M$ exists as an absolutely continuous measure with respect to the prior $\pi$, with Radon-Nikodym derivative
\begin{equation}\label{eq:bayes-likelihood}
  \frac{\ud \pi(\, \cdot \bigmid \vec{u})}{\ud \pi}(f) \propto \pi(\vec{u} \bigmid f) \dfn \rho_M\bigl(\vec{u} - \mathcal{G}(f)\bigr),
\end{equation}
where $\rho_M$ is the density function of $\vec{\varepsilon}_M$, and the implicit proportionality constant is independent of $f$.

As the observation $\vec{u}$ has a continuous distribution on $\real^M$ in our model, the posterior distribution as a conditional probability measure is strictly speaking well-defined only for almost all $\vec{u} \in \real^M$. In our use case (Theorem \ref{th:bayes-ip} below), the density on the right-hand side of \eqref{eq:bayes-likelihood} is a bounded and continuous function of $\vec{u}$, so we may view $\pi(\, \cdot \bigmid \vec{u})$ defined by \eqref{eq:bayes-likelihood} as the canonical posterior for \emph{all} $\vec{u} \in \real^M$. We thus have $\pi(\, \cdot \bigmid \vec{u'}) \longrightarrow \pi(\, \cdot \bigmid \vec{u})$ in distribution (and e.g.~setwise and in the total variation metric etc.) as $\vec{u'} \to \vec{u}$. We refer to \cite[Section 1]{Lasanen-2012-II} for a thorough discussion and literature review concerning the question of existence and uniqueness of posterior distributions in the functional setting.

We will consider forward operators of the form
\begin{equation}\label{eq:forward-pointwise}
  \mathcal{G}(f) \dfn G\left( f(x_1), f(x_2), \cdots, f(x_N), f\right)
\end{equation}
for some fixed $(x_i)_{i=1}^{N} \subset \udomain$, where $G \colon \real^N \times W^{s,p}(\udomain) \to \real^M$ is continuous. 
In order to make sense of the Bayesian posterior with this type of forward operation, $\mathcal{G}$ should be a measurable function on $W^{s,p}(\udomain)$, so strictly speaking we should modify \eqref{eq:forward-pointwise} by replacing the evaluations $f(x_i)$ with e.g.
\[
  \liminf_{n \to \infty} f_{B_\udomain(x_i , 1/n)}
\]
(or zero in case this quantity is not finite), which as a function of $f$ is a pointwise limes inferior of continuous functions on $W^{s,p}(\udomain)$ and thus Borel-measurable. This modification is immaterial for the processes $f^H$, $H \in \nanu \cup \{\infty\}$, as demonstrated by the following Lemma.

\begin{lemma}\label{le:lebesgue-points}
  Let the assumptions on the activation function $\varphi$ be as in Theorem \ref{th:besov-convergence}, and $p \in (d/(d+\lambda), \alpha)$. Then for all $H \in \nanu \cup \{\infty\}$ and $x \in \udomain$,
  \[
    \lim_{r \to 0^+} (f^H)_{B_\udomain(x , r)} = f^H(x)
  \]
  with full probability, and
  \[
    \lim_{r \to 0^+} \E \left[ \left( |f^H - f^H(x)|_{B_\udomain(x , r)} \right)^p \right] = 0
  \]
  uniformly in $H \in \nanu\cup \{ \infty \}$.
\end{lemma}
The proof shares similarities with that of the ``converse'' result Lemma \ref{le:integral-approximations}, and is presented in Appendix \ref{ap:auxiliary}.

\begin{theorem}\label{th:bayes-ip}
  Let $\varphi$ and $(p,s)$ be as in Theorem \ref{th:besov-convergence}, and let $\vec{u} \in \real^N$.
  Assume that the density function $\rho_M$ of the noise term $\vec{\varepsilon}_M$ in \eqref{eq:bayes-ip} is bounded and H\"older-continuous with some exponent $\lambda_\rho \in (0,1]$, and that its support is the entirety of $\real^M$. Assume that the function $G \colon \real^N \times W^{s,p}(\udomain) \to \real^M$ in \eqref{eq:forward-pointwise} is continuous, and uniformly H\"older-continuous with some exponent $\lambda_G \in (0,1]$ in the $\real^N$-component.

  Then, for the inverse problem \eqref{eq:bayes-ip} with forward operator \eqref{eq:forward-pointwise} and prior $\pi^H \dfn \law(f^H)$, we have the posterior convergence
  \[
    \pi^H(\, \cdot \bigmid \vec{u}) \stackrel{w}{\longrightarrow} \pi^\infty(\, \cdot \bigmid \vec{u})
  \]
  in $\pmspace(W^{s,p}(\udomain))$ as $H \to \infty$.
\end{theorem}

The assumptions on the noise density $\rho_M$ are fairly general, and they admit e.g.~Gaussian and component-wise i.i.d symmetric $\alpha$-stable vectors.

\begin{proof}
We first briefly note that the posterior distribution, given observations $\vec{u} \in \real^M$, exists in the sense of \eqref{eq:bayes-likelihood}, and that this is where the assumption about $\rho_M$ being bounded and supported on the entirety of $\real^M$ is needed. Namely, these assumptions quarantee that the posterior distribution is well-defined in the sense that
\begin{equation}\label{eq:posterior-well-defined}
  0 < \int_{W^{s,p}(\udomain)} \rho_M\bigl(\vec{u} - \mathcal{G}(f)\bigr) \, \pi^H(\ud f) < \infty
\end{equation}
for all $\vec{u} \in \real^M$ and $H \in \nanu\cup\{\infty\}$. This existence result can be found in \cite[Section 5.1]{Lasanen-2012-II} under the condition that the function
\[
  W^{s,p}(\udomain) \times \real^M \owns (f,\vec{u}) \mapsto \rho_M\bigl(\vec{u} - \mathcal{G}(f)\bigr) \in [0,\infty)
\]
is jointly Borel-measurable, as is the case under our assumptions.

We introduce the auxiliary notation
\[
  \pi_y(\vec{u} \bigmid f) \dfn \rho_M\bigl( \vec{u} - G(y,f) \bigr)
\]
for $y \in \real^N$ and $f \in W^{s,p}(\udomain)$. It follows from the assumptions that the function
\begin{equation} \label{eq:likelihood-lipschitz}
  \real^N \owns y \mapsto \pi_y(\vec{u} \bigmid f) \in [0,\infty)
\end{equation}
is bounded and H\"older-continuous with exponent $\lambda' \dfn \lambda_\rho \lambda_G \in (0,1]$ uniformly in $f$.

Recall from \eqref{eq:bayes-likelihood} that the unnormalized posterior $\gamma^H( \, \cdot \bigmid \vec{u})$ can for $H \in \nanu\cup\{\infty\}$ and bounded and continuous test functions $F \colon W^{s,p}(\udomain) \to \real$ be expressed by
\[
  \gamma^H(F \bigmid \vec{u}) = \int_{W^{s,p}(\udomain)} F(f) \pi(\vec{u} \bigmid f) \law(f^H)(\ud f),
\]
where
\[
  \pi(\vec{u} \bigmid f) = \pi_{(f(x_1),\cdots,f(x_N))}(\vec{u} \bigmid f)
\]
is the likelihood associated with the forward operator \eqref{eq:forward-pointwise}. For $r > 0$, we write
\[
  \pi^r(\vec{u} \bigmid f) \dfn \pi_{(f_{B_{\udomain}(x_1,r)}, \cdots, f_{B_{\udomain}(x_N,r)})}(\vec{u} \bigmid f).
\]
With fixed $r$, this is by assumption a bounded and continuous function of $f \in W^{s,p}(\udomain)$.

By the continuity properties of the function \eqref{eq:likelihood-lipschitz}, we have
\begin{align}
  & \left|\gamma^H(F \bigmid \vec{u}) - \gamma^\infty(F \bigmid \vec{u})\right|
  \notag
  \\
  \leq & \left| \int_{W^{s,p}(\udomain)} F(f) \left(\pi(\vec{u} \bigmid f) - \pi^r(\vec{u} \bigmid f) \right)\law(f^H)(\ud f) \right|
  \notag
  \\
  & \qquad + \left| \int_{W^{s,p}(\udomain)} F(f) \left(\pi(\vec{u} \bigmid f) - \pi^r(\vec{u} \bigmid f) \right)\law(f^\infty)(\ud f) \right|
  \notag
  \\
  & \qquad + \left| \int_{W^{s,p}(\udomain)} F(f) \pi^r(\vec{u} \bigmid f)\left\{\law(f^H)(\ud f) - \law(f^\infty)(\ud f) \right\}\right|
  \notag
  \\
  \lesssim & \sum_{i=1}^{N} \E \left[ \left( |f^H - f^H(x_i)|_{B_\udomain(x_i , r)} \right)^{p\land \lambda'} \right] \label{eq:posterior-lebesgue-1}
  \\
  & \qquad + \sum_{i=1}^{N} \E \left[ \left( |f^\infty - f^\infty(x_i)|_{B_\udomain(x_i , r)} \right)^{p \land \lambda'} \right]
  \label{eq:posterior-lebesgue-2}
  \\
  & \qquad + \left| \int_{W^{s,p}(\udomain)} F(f) \pi^r(\vec{u} \bigmid f)\left\{\law(f^H)(\ud f) - \law(f^\infty)(\ud f) \right\}\right|.
  \label{eq:posterior-r-approximation}
\end{align}
with some multiplicative constant independent of $H$ and $r > 0$. Given $\epsilon > 0$, we can by Lemma \ref{le:lebesgue-points} take $r$ so that the terms \eqref{eq:posterior-lebesgue-1} and \eqref{eq:posterior-lebesgue-2} are  $< \epsilon/2$ for all $H$. For this $r$, the term \eqref{eq:posterior-r-approximation} converges to zero as $H \to \infty$ by Theorem \ref{th:besov-convergence}.

This shows that
\[
  \lim_{H\to\infty} \gamma^H(F \bigmid \vec{u}) = \gamma^\infty(F \bigmid \vec{u})
\]
for all bounded and continuous $F \colon W^{s,p}(\udomain) \to \real$, and by \eqref{eq:posterior-well-defined}, the same convergence therefore holds for the normalized posteriors $\pi^H(\, \cdot \bigmid \vec{u})$.
\end{proof}

\section{Deeper architectures}\label{se:deeper}

A very natural question is whether the results in the previous sections, particularly the ones concerning functional convergence and Baye\-sian posterior consistency (Theorems \ref{th:besov-convergence} and \ref{th:bayes-ip}), generalize to perceptrons with more than one hidden layer. Here we establish this generalization for the case of Lipschitz-continuous activation functions, which allows for a convenient recursion of the sample regularity analysis through the hidden layers. Similar results could perhaps be attained for more general activations through a more sophisticated analysis of the distributions of the neurons at each hidden layer, but this is beyond the scope of this paper.

To set up the result, denote by $L \in \{2, 3, 4, \cdots\}$ the fixed number of hidden layers in our neural network, by $H_0 \dfn d \in \nanu$ the input dimension, by $H_{L+1} \dfn 1$ the output dimension and by $\vec{H} \dfn (H_1, H_2, \cdots, H_L) \in \nanu^L$ the widths of the hidden layers. Let $\vec{U} \in \real^{H_1 \times H_0}$ be a random matrix with i.i.d components distributed as $\stable{\alpha}{\sigma_u}$, for $\ell \in 1{:}L \dfn \{1, 2, \cdots, L\}$ let $\vec{V}^{(\ell)} \in \real^{H_{\ell+1} \times H_{\ell}}$ be a random matrix with i.i.d components distributed as $\stable{\alpha}{\sigma_v}$, and for $\ell \in 1{:}(L+1)$ let $\vec{a}^{(\ell)} \in \real^{H_\ell}$ be a random vector with i.i.d components distributed as $\stable{\alpha}{\sigma_a}$.

Our neural network is then defined sequentially through the layers as follows. Let
\[
  f^{(1)}(x) \dfn \left( f_1^{(1)}(x), \cdots, f_{H_1}^{(1)}(x) \right) \dfn\vec{U} x + \vec{a}^{(1)} \in \real^{H_1},
\]
and for $\ell \in 1{:}L$, let
\[
  f^{(\ell+1)}(x) \dfn \left( f_1^{(\ell+1)}(x), \cdots, f_{H_{\ell+1}}^{(\ell+1)}(x) \right) \dfn H_{\ell}^{-1/\alpha} \vec{V}^{(\ell)} \varphi\left( f^{(\ell)}(x) \right) + \vec{a}^{(\ell+1)} \in \real^{H_{\ell+1}},
\]
where $\varphi(\cdot) \in \real^{H_\ell}$ is interpreted as an elementwise action on a vector. We denote the ultimate scalar output as
\[
  f^{\vec{H}}(x) \dfn f^{(L+1)}(x). 
\]

The finite-dimensional convergence of these neural networks under some mild assumptions (satisfied by the conditions of Theorem \ref{th:deep} below) has been obtained in \cite[Theorem 2]{FFP23-deep-stable} and \cite[Theorem 3.1 and Section 4]{JLLY23}, in the sense that there exists an $\alpha$-stable process $(f^\infty(x))_{x \in \real^d}$ such that
\begin{equation}\label{eq:deep-convergence-finite-dims}
  \bigl( f^{\vec{H}}(x_1), f^{\vec{H}}(x_2), \cdots, f^{\vec{H}}(x_n)\bigr) \stackrel{d}{\longrightarrow} \bigl( f^\infty(x_1), f^\infty(x_2), \cdots, f^\infty(x_n)\bigr)
\end{equation}
as $\min(H_1, H_2, \cdots, H_L) \to \infty$ for all $\{x_1, x_2, \cdots, x_n\} \subset \real^d$. The cited results also describe a recursive formula for the finite-dimensional spectral measures of $f^{\infty}$, which we omit here.

Let us note that the convergence $\eqref{eq:deep-convergence-finite-dims}$ also holds under some additional hypotheses if the stability index $\alpha$ were to vary between the hidden layers. This however complicates both the analysis and the permissible parameter range for $\beta$ (see the discussion in \cite[Section 4]{JLLY23}), so we will only consider uniformly $\alpha$-stably weighted networks in the following result.

\begin{theorem}\label{th:deep}
Let $\udomain$ and $\varphi$ be as in Theorem \ref{th:besov-convergence}, with $\varphi$ Lipschitz-continuous (i.e.~$\lambda = 1$), and let $\alpha \in (d/(d+1), 2)$. Then the statements of Theorems \ref{th:besov-convergence} and \ref{th:bayes-ip} continue to hold with $f^{\vec{H}}$ in place of $f^H$, and limiting process $f^{\infty}$ for $\min(H_1,\cdots,H_L) \longrightarrow \infty$ described by \eqref{eq:deep-convergence-finite-dims}.
\end{theorem}

\begin{proof}
Key to much of our analysis in the shallow regime is the uniform modulus of continuity estimate in Proposition \ref{pr:measurable-sample-functions} (\ref{en:modulus-of-continuity}). We will establish a weaker version of this estimate for deep networks with Lipschitz activation, which nevertheless is qualitatively close enough in a certain sense.

We begin by noting that for each $i \in 1{:}H_2$ and all $x$, $y \in \udomain$,
\[
  f^{(2)}_i(x) - f^{(2)}_i(y)
\]
is distributed as $f^H(x) - f^H(y)$, where $f^H$ is the shallow network considered in Section \ref{se:sample-regularity}. Thus, by Proposition \ref{pr:measurable-sample-functions} (\ref{en:modulus-of-continuity}), we have
\begin{equation}\label{eq:deep-first-hidden-layer}
  \E\left[ |f^{(2)}_i(x) - f^{(2)}_i(y)|^p \right] \lesssim |x - y|^{p - \epsilon}
\end{equation}
for all $x$, $y \in \udomain$, $p \in (0, \alpha)$ and $\epsilon \in (0,p)$ close to $0$ (to be specified below). The implicit constant here depends (among other things) on $\alpha$, $p$, $\epsilon$ and $\udomain$, but not on $(x,y)$ nor the widths $\vec{H}$.

Let then $c \in (\beta, 1)$ be close to $1$ (again, to be specified below). By the continuity and growth criteria on $\varphi$, it is immediate that
\begin{equation}\label{eq:almost-lipschitz}
  |\varphi(x) - \varphi(y)| \lesssim |x - y|^c
\end{equation}
for all $x$, $y \in \real$, with an implicit constant that does not depend on $(x,y)$. With this estimate, we can for any $i \in 1{:}H_3$ carry out the following estimates using the freezing lemma, as in the proof of Proposition \ref{pr:measurable-sample-functions}:
\begin{align*}
  & \E\left[ |f^{(3)}_i(x) - f^{(3)}_i(y)|^p \right]
  \\
  = \, & \E\left[ \left| \sum_{j=1}^{H_2} \frac{\varphi\bigl(f^{(2)}_j(x)\bigr) - \varphi\bigl(f^{(2)}_j(y)\bigr)}{H_2^{1/\alpha}} \vec{V}^{(2)}_{i,j} \right|^p \right]
  \\
  = \, & \E\left[ |\vec{V}^{(2)}_{1,1} |^p\right] \E\left[ \left( \frac{1}{H_2} \sum_{j=1}^{H_2} \bigl|\varphi\bigl(f^{(2)}_j(x)\bigr) - \varphi\bigl(f^{(2)}_j(y)\bigr)\bigr|^\alpha \right)^{p/\alpha} \right]
  \\
  \leq \, & \E\left[ |\vec{V}^{(2)}_{1,1} |^p\right] \left( \frac{1}{H_2} \sum_{j=1}^{H_2} \E\left[ \bigl| \varphi\bigl(f^{(2)}_j(x)\bigr) - \varphi\bigl(f^{(2)}_j(y)\bigr) \bigr|^\alpha \right]\right)^{p/\alpha}
  \\
  \lesssim \, & \left( \frac{1}{H_2} \sum_{j=1}^{H_2} \E\left[ |f^{(2)}_j(x) - f^{(2)}_j(y)|^{c\alpha} \right] \right)^{p/\alpha} \lesssim |x-y|^{cp - \epsilon(p/\alpha) },
\end{align*}
where in the second-to-last inequality we used \eqref{eq:almost-lipschitz}, and in the last one we used \eqref{eq:deep-first-hidden-layer} with $c\alpha < \alpha$ in place of $p$ -- here $\epsilon$ has to be small enough so that $\epsilon < c\alpha$. Since $c \in (\beta, 1)$ was arbitrary, and $p/\alpha < 1$, we can take $c$ so that $cp - \epsilon(p/\alpha) > p - \epsilon$, yielding
\[
  \E\left[ |f^{(3)}_i(x) - f^{(3)}_i(y)|^p \right] \lesssim |x - y|^{p - \epsilon}
\]
for all $i \in 1{:}H_3$, which is of the same form as \eqref{eq:deep-first-hidden-layer} with $f^{(3)}$ in place of $f^{(2)}$.

We can induct on this idea in an obvious manner, culminating in
\begin{equation}\label{eq:deep-modulus-of-continuity}
  \E\left[ |f^{\vec{H}}(x) - f^{\vec{H}}(y)|^p \right] \lesssim |x - y|^{p - \epsilon}
\end{equation}
for all $x$, $y \in \udomain$ and any sufficiently small $\epsilon > 0$. To be clear, the implicit multiplicative constant depends among other things on $L$, $\alpha$, $p$, $\epsilon$ and $\udomain$, but \emph{not} on $(x,y)$ nor the widths $\vec{H}$. In particular, the uniformity with respect to $\vec{H}$ allows us to extend \eqref{eq:deep-modulus-of-continuity} to the limit process $f^{\infty}$ as in the proof of Proposition \ref{pr:measurable-sample-functions}.

The idea is then to use \eqref{eq:deep-modulus-of-continuity} instead of Proposition \ref{pr:measurable-sample-functions} (\ref{en:modulus-of-continuity}) with a sufficiently small $\epsilon > 0$ whenever needed -- this will always be possible due to the open-ended nature of the parameter ranges appearing in our results. In fact, we have done this kind of parameter nudging even in the shallow regime in the proofs of Lemmas \ref{le:integral-approximations} and \ref{le:lebesgue-points} in the Lipschitz case, in order to avoid dealing with the logarithmic term in Proposition \ref{pr:measurable-sample-functions} (\ref{en:modulus-of-continuity}).

Let us sketch a proof of the counterpart of Proposition \ref{pr:besov-integrability} (\ref{en:besov-integrability}) as an example.

First, a measurable version of each $f^{\vec{H}}$, $\vec{H} \in \nanu^L \cup \{\infty\}$, exists as in Proposition \ref{pr:measurable-sample-functions} (\ref{en:measurable-version}), which justifies the use of Fubini's theorem in what follows. Let $s \in (0,1)$ and $p \in (d/(d+s), \alpha)$. That
\[
  \sup_{\vec{H} \in \nanu^L} \E\left[ \left\| f^{\vec{H}} \right\|^p_{L^p(\udomain)} \right]
  = \sup_{\vec{H} \in \nanu^L} \int_{\udomain} \E\left[ \left| f^{\vec{H}}(x) \right|^p\right] \ud x
  < \infty
\]
can be proven by inductively estimating the integrand through the hidden layers as above, using the sub-polynomial growth of $\varphi$, and this uniform estimate extends to $f^{\infty}$ in the same way as \eqref{eq:deep-modulus-of-continuity}.

Then, for $\vec{H} \in \nanu^L \cup \{\infty\}$, we can use \eqref{eq:deep-modulus-of-continuity} to estimate
\begin{align*}
  \E\left[ \left\| f^{\vec{H}} \right\|^p_{\dot{W}^{s,p}(\udomain)} \right]
  & = \int_{\udomain \times \udomain} |x-y|^{-sp - d} \E\left[ |f^{\vec{H}}(x) - f^{\vec{H}}(y)|^p \right] \ud x \ud y
  \\
  & \lesssim \int_{\udomain \times \udomain} |x-y|^{-sp-d + p - \epsilon} \ud x \ud y,
\end{align*}
with an implicit constant independent of $\vec{H}$. Since $s < 1$, the parameter $\epsilon > 0$ can be taken smaller than $p-sp$, so that the latter double integral is finite.
\end{proof}

\appendix
\section{Properties of Sobolev-Slobodeckij spaces}\label{ap:properties-besov}
Here we collect some useful properties of the Sobolev-Slobodeckij spaces $W^{s,p}$, along with references.

Let $\udomain \subset \real^d$ be a bounded domain in the Euclidean space with $C^\infty$-smooth boundary. It is not hard to see that the domain $\udomain$ is then also \emph{Ahlfors $d$-regular}, in the sense that
\begin{equation}\label{eq:ahlfors-regular}
  m_d\left( B_\udomain(x, r) \right) \approx r^d
\end{equation}
for all $x \in \udomain$ and $r \in (0, \mathop{\mathrm{diam}}(\udomain)]$ with some multiplicative constants indepedent of $(x,r)$, where $m_d$ and $B_\udomain$ stand for the Lebesgue $d$-measure and the Euclidean ball restricted to $\udomain$.

For $s \in (0,1)$ and $p > d/(d+s)$, we then recall the homogeneous function space $\dot{W}^{s,p}(\udomain)$, defined in terms of the quasinorm (modulo additive constants)

\[
  \|f\|_{\dot{W}^{s,p}(\udomain)} \dfn \left( \int_{\udomain \times \udomain} \frac{|f(x)-f(y)|^p}{|x-y|^{sp+d}} \ud x \ud y\right)^{1/p}.
\]
The space $W^{s,p}(\udomain)$ is then defined as $L^p(\udomain) \cap \dot{W}^{s,p}(\udomain)$, with proper quasinorm
\begin{equation}\label{eq:besov-quasinorm}
  \|f\|_{W^{s,p}(\udomain)} \dfn \|f\|_{L^p(\udomain)} + \|f\|_{\dot{W}^{s,p}(\udomain)}.
\end{equation}

\begin{proposition}\label{pr:besov-properties}
Let $s \in (0,1)$ and $p > d/(d+s)$. We have the following structural properties for the space $W^{s,p}(\udomain)$.

\begin{enumerate}[(i)]
  \item\label{en:equivalent-quasinorms}
  The space $W^{s,p}(\udomain)$ coincides with the space of distributions $B^s_{p,p}(\udomain)$ considered in \cite[Chapter 3]{triebel-I} and the space $\mathcal{F}^s_{p,p}(\udomain)$ of locally integrable functions considered in \cite[Definition 5.1]{SS17}, all with equivalent quasinorms.

  \item \label{en:L1-embedding}
  For $d/(d+s) < p < 1$, the the functions in $W^{s,p}(\udomain)$ are integrable, and the term $\|f\|_{L^p(\udomain)}$ in \eqref{eq:besov-quasinorm} can be replaced with $\|f\|_{L^1(\udomain)}$, resulting in the same space with equivalent quasinorms.

  \item \label{en:jawerth-embedding}
  For $p' > p$ and $s' \in (0,s)$ such that $s' - d/p' = s-d/p$, we have
  \begin{equation}\label{eq:jawerth-embedding}
    W^{s,p}(\udomain) \subset W^{s',p'}(\udomain)
  \end{equation}
  with a continuous embedding (it automatically holds that $p' > d/(s'+d)$).

  \item \label{en:rellich-kondrachov}
  For $p' > p$ and $s' \in (0,s)$ such that $s' - d/p' > s-d/p$, the embedding \eqref{eq:jawerth-embedding} is compact (it automatically holds that $p' > d/(s'+d)$).

  \item \label{en:separable-metric}
  Endowed with
  \[
    (f, g) \mapsto \|f-g\|_{W^{s,p}(\udomain)}^{p\land 1},
  \]
  $W^{s,p}(\udomain)$ is a complete and separable metric space.
\end{enumerate}
\end{proposition}

The classic monograph \cite[Chapter 3]{triebel-I} contains many of the results mentioned above. In the interest of accessibility, we present here a more precise list of sources for each of the claims in the above Proposition.

\begin{enumerate}[(i)]
  \item That $W^{s,p}(\udomain)$ coincides with the distribution space $B^s_{p,p}(\udomain)$ is contained in \cite[Proposition 3.4.2]{triebel-I}. Concerning the space $\mathcal{F}^s_{p,p}(\udomain)$, we first note that $W^{s,p}(\udomain)$ by definition agrees with the space $\dot{B}^s_{p,p}(\udomain) \cap L^p(\udomain)$, where the homogeneous space $\dot{B}^s_{p,p}(\udomain)$ is as defined in \cite[Definition 1.1]{GKZ13}. Then by \cite[Theorem 1.2 and Definition 3.1]{GKZ13}, $\dot{B}^s_{p,p}(\udomain)$ coincides with a certain function space $\dot{M}^s_{p,p}(\udomain)$ defined in terms of fractional gradients. By \cite[Proposition 3.1]{BSS18}, this space in turn coincides with the homogeneous space $\dot{\mathcal{F}}^s_{p,p}(\udomain)$, and by \cite[Proposition 5.2]{SS17} we have $\dot{\mathcal{F}}^s_{p,p}(\udomain)\cap L^p(\udomain) = \mathcal{F}^s_{p,p}(\udomain)$ with equivalent quasinorms.

  \item This result for the distribution space $B^s_{p,p}(\udomain)$ is contained in \cite[Theorem 1.118 (ii)]{triebel-III} (see also the discussion in \cite[Remark 1.117]{triebel-III}).

  \item[(iii--iv)] These embeddings can be found in \cite[Remark 1.96 and Theorem 1.97]{triebel-III} for the distribution spaces $B^s_{p,p}(\udomain)$ and $B^{s'}_{p',p'}(\udomain)$.

  \item[(v)] That the function of $(f,g)$ in question is in fact a metric is an easy consequence of the sub-additivity of the function $x \mapsto |x|^{r \land 1}$ for any $r > 0$. Completeness and separability are more or less straightforward consequences of the corresponding properties of the distribution space $B^s_{p,p}(\real^d)$ on the Euclidean space, which are inherited by the localized space $B^s_{p,p}(\udomain)$; see \cite[Remark 1.96]{triebel-III}.
\end{enumerate}

\section{Auxiliary results}\label{ap:auxiliary}

Recall from Section \ref{se:sample-regularity} the notation
\[
  f_A \dfn \frac{1}{m_d(A)} \int_A f(x) \ud x
\]
for all $f \colon \udomain \to \real$ and $A \subset \udomain$ such that the right-hand side above is well-defined (in particular, it is for $f \in L^1(\udomain)$ and measurable $A \subset \udomain$ with $m_d(A) > 0$).

Before stating the first results of this section, we recall a family of discrete convolution operators on $L^1(\udomain)$ which have been studied in the context of first-order Besov spaces in e.g.~\cite{BSS18,SS17,HKT-2017}, and which we will apply extensively below. For this purpose, we fix for every $n \in \nanu$ a maximal set of points $(x^n_i)_{i=1}^{\kappa(n)} \subset \udomain$ such that $|x^n_i - x^n_j| \geq 2^{-n-1}$ for all $i \neq j$ (it can be shown that $\kappa(n) \approx 2^{dn}$ with implicit constants independent of $n$). Let
\[
  \mathcal{B}^n \dfn \left( B^n_1, \cdots, B^n_{\kappa(n)} \right) \dfn \left( B_\udomain(x^n_1, 2^{-n}), \cdots, B_\udomain(x^n_{\kappa(n)}, 2^{-n}) \right),
\]
and for $f \in L^1(\udomain)$
\[
  \pi_{\mathcal{B}^n}(f) \dfn \bigl( f_{B^n_1}, \cdots, f_{B^n_{\kappa(n)}}\bigr) \in \real^{\kappa(n)}.
\]

It is easy to see that the balls in $\mathcal{B}^n$ have uniformly (with respect to $n$) bounded overlap, and that there exists a family of Lipschitz continuous functions $(\psi^n_i)_{i=1}^{\kappa(n)}$ on $\udomain$ such that $\mathop{\mathrm{supp}} \psi^n_i \subset B^n_i$, $\psi^n_i \geq 0$,
\[
  \sum_{i=1}^{\kappa(n)} \psi^n_i \equiv 1
\]
and that the Lipschitz constants of the $\psi^n_i$ are bounded by a constant (indepdent of $n$) times $2^n$. We refer to e.g.~\cite[Definition 3.4]{SS17} and the discussion therein.

\begin{definition}\label{de:discrete-convolution}
  \begin{enumerate}[(i)]
    \item
    For $n \in \nanu$ and $v \in \real^{\kappa(n)}$, define the function $T^n v \colon \real^{\kappa(n)} \to L^1(\udomain)$ by
    \[
      T^n v = \sum_{i=1}^{\kappa(n)} v(i) \psi^n_i.
    \]

    \item
    For $n \in \nanu$ and $f \in L^1(\udomain)$, define
    \[
      T^n f \dfn T^n\bigl( \pi_{\mathcal{B}^n} (f) \bigr).
    \]
  \end{enumerate}
\end{definition}

\begin{proposition}\label{pr:discrete-convolution}
  Let $s \in (0,1)$ and $p > d/(d+s)$.

  \begin{enumerate}[(i)]
    \item\label{en:discrete-conv-sequences}
    For each $n$, the operator $T^n$ is continuous from $\real^{\kappa(n)}$ to $W^{s,p}(\udomain)$.

    \item\label{en:discrete-conv-functions}
    The convergence
    \[
      \lim_{n \to \infty} T^n f = f \quad
    \]
    holds in the metric of $W^{s,p}(\udomain)$ for all $f \in W^{s,p}(\udomain)$.
  \end{enumerate}
\end{proposition}

\begin{proof}
Part (i) is a simple computation using the Lipschitz continuity and boundedness properties of the $\psi^n_i$'s. By linearity, it suffices to check that
\[
  \| T^n v\|_{W^{s,p}(\udomain)}^{p\land 1} \lesssim \kappa(n) 2^{sn/(p \lor 1)} \max_{1 \leq i \leq \kappa(n)} |v(i)|^{p \land 1}
\]
with an implicit multiplicative constant independent of $v$ and $n$. Part (ii) can be found in \cite[Proposition 5.4]{SS17}.
\end{proof}

For the crucial measurability property in Proposition \ref{pr:besov-integrability} (\ref{en:besov-process}), we need the following description of closed balls in $W^{s,p}(\udomain)$, which gives an almost complete description of the topology of $W^{s,p}(\udomain)$ in terms of the operators $\pi_{\mathcal{B}^n}$ introduced above. Note that by Proposition \ref{pr:besov-properties} (\ref{en:L1-embedding}) in Appendix \ref{ap:properties-besov}, each $\pi_{\mathcal{B}^n}$ is continuous from $W^{s,p}(\udomain)$ to $\real^{\kappa(n)}$.

\begin{lemma}\label{le:auxiliary-closed-ball}
Let $s \in (0,1)$ and $p > d/(d+s)$. For $f \in W^{s,p}(\udomain)$, $\varepsilon > 0$ and $n \in \nanu$, write $V^n(f,\varepsilon)$ for the closed set
\[
  \left\{ v \in \real^{\kappa(n)} \given \| T^n v - T^n f\|_{W^{s,p}(\udomain)}^{p\land 1} \leq \varepsilon \right\} \subset \real^{\kappa(n)}.
\]
Then the closed $W^{s,p}(\udomain)$-ball centered at $f \in W^{s,p}(\udomain)$ with radius $\varepsilon$ can be expressed as
\begin{equation}\label{eq:besov-closed-ball}
  \overline{B}_{W^{s,p}(\udomain)}(f, \varepsilon) = \bigcap_{m = 1}^{\infty} \bigcup_{n=m}^{\infty} \pi_{\mathcal{B}^n}^{-1}\left(V^n(f, \varepsilon_m) \right),
\end{equation}
where $(\varepsilon_m)_{m\in\nanu}$ is any strictly decreasing sequence converging to $\varepsilon$.
\end{lemma}
That the sets $V^n(f,\varepsilon)$ in the statement are closed follows from Proposition \ref{pr:discrete-convolution} (\ref{en:discrete-conv-sequences}) above.
\begin{proof}
We have by Proposition \ref{pr:discrete-convolution} (\ref{en:discrete-conv-functions}) that
\begin{align*}
  \|g - f\|_{W^{s,p}(\udomain)}^{p \land 1}
  & = \lim_{n\to \infty} \|T^n g - T^n f\|_{W^{s,p}(\udomain)}^{p \land 1}
\end{align*}
for all $g$, $f \in W^{s,p}(\udomain)$. The inclusions between the left-hand and right-hand sides of \eqref{eq:besov-closed-ball} in both directions follow straightforwardly from this.
\end{proof}

\begin{proof}[Proof of Proposition \ref{pr:besov-integrability} (\ref{en:besov-process}), continuation]
Write $F \colon \Omega^H \to W^{s,p}(\udomain)$ for the mapping $\omega \mapsto f^H(\omega, \cdot)$. By Lemma \ref{le:auxiliary-closed-ball}, the closed $W^{s,p}(\udomain)$-ball at $f$ with radius $\varepsilon$ can be written as \eqref{eq:besov-closed-ball}, where the $V^n(f,\varepsilon_m)$ are closed subsets as in the statement of the Lemma. Thus,
\begin{align*}
  F^{-1} \overline{B}_{W^{s,p}(\udomain)}(f, \varepsilon)
  & = \bigcap_{m = 1}^{\infty} \bigcup_{n=m}^{\infty} F^{-1} \pi_{\mathcal{B}^n}^{-1}\left(V^n(f, \varepsilon_m) \right)
  \\
  & = \bigcap_{m = 1}^{\infty} \bigcup_{n=m}^{\infty} \left\{ \omega \in \Omega^H \given \pi_{\mathcal{B}^n}\left( f^H(\omega, \cdot) \right) \in V^n(f, \varepsilon_m)\right\}.
\end{align*}

Recall the measurability of the process $f^H$, established in Proposition \ref{pr:measurable-sample-functions} (\ref{en:measurable-version}). The measurability part of Fubini's theorem, i.e.~that integrals with respect to one coordinate are measurable with respect to the other, can be localized to integrals over the balls in $\mathcal{B}^n$, meaning $\omega \mapsto \pi_{\mathcal{B}^n}( f^H(\omega, \cdot) ) \in \real^{\kappa(n)}$ is an $\mathcal{S}^H$-measurable random vector for all $n$. Each set in the inner union above is therefore in $\mathcal{S}^H$, and hence so is $F^{-1} \overline{B}_{W^{s,p}(\udomain)}(f, \varepsilon)$.
\end{proof}

A key ingredient of the proof of the main result, similar in flavor to the measurability proof above, is the following observation, which essentially implies that after establishing the relative compactness of the laws of the processes $f^H$, it suffices to check the convergence of local averages to those of $f^{\infty}$.

\begin{proposition}\label{pr:local-averages}
  Let $s \in (0,1)$ and $p > d/(d+s)$. Then the collection $\Pi$ of local averages, i.e.~the mappings from $W^{s,p}(\udomain)$ to finite-dimensional Euclidean spaces of the form
  \[
    \pi_{B_1, \cdots, B_n} \dfn f \mapsto \left( f_{B_1} , \cdots, f_{B_n} \right) \in \real^n
  \]
  for some ordered collection $(B_i)_{1 \leq i \leq n}$ of balls in $\udomain$, separates $\pmspace(W^{s,p}(\udomain))$, in the sense that if $P$ and $Q$ are Borel probability measures on $W^{s,p}(\udomain)$ such that $P\circ \pi^{-1} = Q\circ \pi^{-1}$ for all $\pi \in \Pi$, then $P = Q$.
\end{proposition}

\begin{proof}
It suffices to show (see \cite[Example 1.3]{billingsley} and the discussion therein) that the class of sets
\begin{equation} \label{eq:separating-sets}
  \left\{ \pi^{-1}_{\mathcal{B}} (V) \given \mathcal{B} \text{ finite collection of balls}, \, V \subset \real^{\#\mathcal{B}} \text{ closed}\right\}
\end{equation}
is a $\pi$-system, and that the $\sigma$-field generated by this class coincides with the Borel $\sigma$-field of $W^{s,p}(\udomain)$.

The $\pi$-system property is immediate from the definition. The mappings in $\Pi$ are continuous on $W^{s,p}(\udomain)$ by Proposition \ref{pr:besov-properties} (\ref{en:L1-embedding}), so it suffices to show that every closed ball in $W^{s,p}(\udomain)$ is contained in the $\sigma$-field generated by \eqref{eq:separating-sets}. This follows directly from Lemma \ref{le:auxiliary-closed-ball}.
\end{proof}

We now turn to the proofs of Lemmas \ref{le:integral-approximations} and \ref{le:lebesgue-points}. For this purpose, we introduce another approximation scheme, this time a probabilistic one for the processes $f^H$.

For all $n \in \nanu$, write $Q^n_j$, $1 \leq j \leq \kappa(n)$, for the collection of the standard semi-open dyadic cubes in $\real^d$ with side length $2^{-n}$ that intersect $\udomain$, and pick arbitrary $x^n_j \in Q^n_j \cap \udomain$. Note that $\kappa(n)$ is not necessarily the same number as the corresponding one we used in the definition of the operators $T^n$ earlier in this appendix, but we still have $\kappa(n) \approx 2^{dn}$ with multiplicative constants independent of $n$. 
For $H \in \nanu\cup\{\infty\}$ and $n \in \nanu$, write $f^{H,n}$ for the process on $\udomain$ defined as
\[
  f^{H,n}(x) = \sum_{j=1}^{\kappa(n)} f^H(x^n_j) \boldsymbol{1}\left(x \in Q^n_j\right).
\]

We recall (Proposition \ref{pr:measurable-sample-functions} (\ref{en:measurable-version})) that each process $f^H$ was realized as a measurable version on the product space $\Omega^H \times \udomain$. Write $\mu^H$ for the product measure on this product space. For any given $p \in (0, \alpha)$ we have by Proposition \ref{pr:measurable-sample-functions} (\ref{en:modulus-of-continuity}) and Fubini's theorem
\begin{align*}
  & \int_{\Omega^H \times \udomain} \sum_{n \in \nanu} \left| f^{H}(\omega, x) - f^{H,n}(\omega, x) \right|^p \mu^H\left( \ud (\omega, x) \right)
  \\
  = & \sum_{n \in \nanu} \int_{\udomain} \E\left[ \left| f^{H}(x) - f^{H,n}(x) \right|^p \right] \ud x < \infty,
\end{align*}
which means that the series inside the left-hand side integral converges $\mu^H$-almost surely, and so
\begin{equation}\label{eq:subsequence-ae-convergence}
  \lim_{n \to \infty} f^{H,n}(\omega, x) = f^{H}(\omega, x)
\end{equation}
for $\mu^H$-almost every $(\omega,x)$. It is then not hard to see that for $\P^H$-almost every $\omega$, \eqref{eq:subsequence-ae-convergence} holds for $m_d$-almost every $x \in \udomain$.

\begin{proof}[Proof of Lemma \ref{le:integral-approximations}]
We will show that
\begin{equation}\label{eq:integral-approximations-str}
  \lim_{n \to \infty} \E \left[ \left( \int_{\udomain} \left| f^H(x) - f^{H,n}(x) \right| \ud x \right)^p \right] = 0
\end{equation}
for all $p \in (d/(d+\lambda), \alpha)$, uniformly in $H \in \nanu\cup\{\infty\}$, so the statement of the lemma follows by taking $x^n_j$ as in the definition of the process $f^{H,n}$ above and $\theta^n_j = m_d( Q^n_j \cap A)/m_d(A)$ (possibly discarding the point-scalar pairs $(x^n_j, \theta^n_j)$ where $\theta^n_j = 0)$.

For $p \geq 1$, \eqref{eq:integral-approximations-str} is an immediate consequence of H\"older's inequality, Fubini's theorem and Proposition \ref{pr:measurable-sample-functions} (\ref{en:modulus-of-continuity}). Below we will consider the case $p \in (d/(d+\lambda), 1 \land \alpha)$, where instead of H\"older's inequality we approximate the integral in \eqref{eq:integral-approximations-str} with a countable sum and use the subadditivity of the function $x \mapsto |x|^p$.
For this purpose, we write $\lambda' \dfn \lambda$ in case $\lambda < 1$, and if $\lambda = 1$, pick $\lambda' \in (0,1)$ so that $p > d/(d+\lambda')$.

We begin by estimating
\begin{align*}
  \int_{\udomain} |f^{H,n+1}(x) - f^{H,n}(x)| \ud x
  & = \sum_{j=1}^{\kappa(n)} \sum_{Q^{n+1}_k \subset Q^n_j} |f^H(x^{n+1}_k) - f^H(x^n_j)| m_d(Q^{n+1}_k \cap \udomain)
  \\
  & \leq 2^{-nd} \sum_{j=1}^{\kappa(n)} \sum_{Q^{n+1}_k \subset Q^n_j} |f^H(x^{n+1}_k) - f^H(x^n_j)|,
\end{align*}
and since there is an uniformly bounded number of cubes $Q^{n+1}_k$ contained in each $Q^{n}_j$, we can use subadditivity and the modulus of continuity estimate in Proposition \ref{pr:measurable-sample-functions} (\ref{en:modulus-of-continuity}) we get
\[
  \E \left[ \left( \int_{\udomain} \left| f^{H,n+1}(x) - f^{H,n}(x) \right| \ud x \right)^p \right] \lesssim 2^{n(d - (d+\lambda')p)}
\]
with a multiplicative constant independent of $n$ and $H$. Since $n(d - (d+\lambda')p) < 0$ by assumption, $n+1$ in the integral above can by a telescoping argument be replaced by any $m > n$, and in particular
\begin{equation}\label{eq:integral-approximation-mn}
  \sup_{m > n} \E \left[ \left( \int_{\udomain} \left| f^{H,m}(x) - f^{H,n}(x) \right| \ud x \right)^p \right] \lesssim 2^{n(d - (d+\lambda')p)}
\end{equation}
with again a multiplicative constant independent of $n$ and $H$.

We now recall that the pointwise convergence \eqref{eq:subsequence-ae-convergence} holds $\P^H$-almost surely for $m_d$-almost every $x \in \udomain$. Thus, by Fatou's lemma,
\begin{align*}
  \E \left[ \left( \int_{\udomain} \left| f^H(x) - f^{H,n}(x) \right| \ud x \right)^p \right]
  & \leq \liminf_{m \to \infty} \E \left[ \left( \int_{\udomain} \left| f^{H,m}(x) - f^{H,n}(x) \right| \ud x \right)^p \right]
  \\
  & \leq \sup_{m > n} \E \left[ \left( \int_{\udomain} \left| f^{H,m}(x) - f^{H,n}(x) \right| \ud x \right)^p \right],
\end{align*}
and the latter quantity converges to zero as $n \to \infty$ uniformly in $H$, as noted in \eqref{eq:integral-approximation-mn}.
\end{proof}

\begin{proof}[Proof of Lemma \ref{le:lebesgue-points}]
Let $\lambda' \dfn \lambda$ in case $\lambda < 1$, and if $\lambda = 1$, pick $\lambda' \in (0,1)$ so that $p > d/(d+\lambda')$. We will start by establishing
\begin{equation}\label{eq:lebesgue-uniform}
  \E \left[ \left( |f^H - f^H(x)|_{B_\udomain(x , r)} \right)^p \right] \lesssim r^{\lambda' p}
\end{equation}
for $r \in (0, 1/2)$, with a multiplicative constant independent of $r$ and $H$. As in the proof of Lemma \ref{le:integral-approximations} above, this is for $p \geq 1$ an immediate consequence of H\"older's inequality and Proposition \ref{pr:measurable-sample-functions} (\ref{en:modulus-of-continuity}). Below we will establish this estimate for $p \in (d/(d + \lambda'), 1\land \alpha)$.

Recall the discrete approximation processes $f^{H,n}$, $n \in \nanu$, introduced before the proof of Lemma \ref{le:integral-approximations}. We begin by considering $n$ such that $2^{-n-1} \leq r < 2^{-n}$. It is not hard to see that there are finitely many cubes $Q^n_j$ that intersect $B_{\udomain}(x,r)$, and the number of such cubes has an upper bound independent of $x$ and $r$. Thus, by  Proposition \ref{pr:measurable-sample-functions} (\ref{en:modulus-of-continuity}),
\begin{align*}
  \E \left[ \left( |f^{H,n} - f^H(x)|_{B_\udomain(x , r)} \right)^p \right]
  \lesssim \E \left[ \left( \sum_{Q^n_j \cap B_{\udomain}(x,r) \neq \emptyset} |f^H(x^n_j) - f^H(x)| \right)^p \right] \lesssim 2^{-n \lambda' p}.
\end{align*}
Next, for $m \geq n$, we note that there are $\lesssim 2^{(m-n)d}$ cubes $Q^m_j$ intersecting $B_{\udomain}(x,r)$ with a multiplicative constant independent of $x$ and $r$, so
\begin{align*}
  & \E \left[ \left( |f^{H,m+1} - f^{H,m}(x)|_{B_\udomain(x , r)} \right)^p \right]
  \\
  & \qquad \lesssim  \E \left[ \left( 2^{nd} \sum_{Q^m_j \cap B_{\udomain}(x,r) \neq \emptyset} 2^{-m d} \sum_{Q^{m+1}_i \subset Q^m_j}|f^H(x^m_j) - f^H(x^{m+1}_i)| \right)^p \right] 
  \\
  & \qquad \lesssim 2^{(n-m)d p} 2^{(m-n)d} 2^{-m \lambda' p} = 2^{n(dp - d)} 2^{m(d-(d+\lambda')p)},
\end{align*}
where $d - (d+\lambda')p < 0$ by assumption.

Recall that $f^{H,n}(\omega, x) \to f^H(\omega, x)$ as $n \to \infty$ almost everywhere in the product space $\Omega^H \times \udomain$ with respect to the product measure (see the discussion preceding the proof of Lemma \ref{le:integral-approximations}). 
Combining the estimates established above we have
\begin{align*}
  & \E \left[ \left( |f^H - f^H(x)|_{B_\udomain(x , r)} \right)^p \right] 
  \\
  & \qquad \leq \E \left[ \left( |f^{H,n} - f^H(x)|_{B_\udomain(x , r)} \right)^p \right] + \sum_{m \geq n} \E \left[ \left( |f^{H,m+1} - f^{H,m}(x)|_{B_\udomain(x , r)} \right)^p \right]
  \\
  & \qquad \lesssim 2^{-n \lambda' p} + 2^{n(dp - d)} \sum_{m \geq n} 2^{m(d-(d+\lambda')p)}
  \approx 2^{-n \lambda' p} \approx r^{\lambda' p},
\end{align*}
which is the estimate \eqref{eq:lebesgue-uniform}.

The first claim of the Lemma will follow if we can show
\begin{equation}\label{eq:lebesgue-pointwise}
  \lim_{r \to 0^+} |f^H - f^H(x)|_{B_\udomain(x , r)} = 0
\end{equation}
almost surely. By the Ahlfors regularity condition \eqref{eq:ahlfors-regular}, for any $r > 0$ we have
\[
  |f^H - f^H(x)|_{B_\udomain(x , 2^{-n-1})} \lesssim |f^H - f^H(x)|_{B_\udomain(x , r)} \lesssim |f^H - f^H(x)|_{B_\udomain(x , 2^{-n})}
\]
for $n \in \integer$ such that $2^{-n-1} \leq r < 2^{-n}$, so it suffices to show \eqref{eq:lebesgue-pointwise} for $r = 2^{-n} \to 0$. By \eqref{eq:lebesgue-uniform}, we have
\[
  \E \left[ \left( \sum_{n \geq 1} |f^H - f^H(x)|_{B_\udomain(x , 2^{-n})} \right)^{p \land 1} \right] \lesssim \sum_{n \geq 1} 2^{-n \lambda' (p \land 1)} < \infty,
\]
which means that the series inside the left-hand side expectation has to converge almost surely.
\end{proof}

\section*{Acknowledgments}
The author would like to thank Lassi Roininen and Petteri Piiroinen for helpful discussions concerning the subject. The work was supported by the Research Council of Finland through the Flagship of Advanced Mathematics for Sensing, Imaging and Modelling, and the Centre of Excellence of Inverse Modelling and Imaging (decision numbers 359183 and 353095).

\end{document}